\documentclass[a4paper,12pt]{article}

\usepackage{amsmath ,amsfonts,  amssymb, amscd, amsthm} 
\usepackage{bbm}
\usepackage{mathrsfs}
\usepackage[font=small,skip=0pt]{caption}

\usepackage{csquotes, color}
\usepackage{colortbl}
\usepackage[numbers]{natbib}
\usepackage{verbatim}
\usepackage[toc,page]{appendix}
\usepackage{caption}
\usepackage{textcomp}
\usepackage[inline]{enumitem}
\usepackage{tcolorbox}
\usepackage{cancel}
\usepackage{relsize}
\usepackage{bm}

\usepackage{authblk}

\newtheorem{theorem}{Theorem}
\newtheorem*{theorem*}{Theorem}
\newtheorem{corollary}{Corollary}[theorem]

\newtheorem{proposition}{Proposition}

\newtheorem*{definition*}{Definition}

\newtheorem*{example*}{Example}
\newtheorem{remark}{Remark}
\newtheorem*{conditional*}{Conditional}

\newtheorem*{remark*}{Remark}

\newcommand\Tstrut{\rule{0pt}{3.75ex}}       
\newcommand\Bstrut{\rule[-3.5ex]{0pt}{0pt}} 
\newcommand{\TBstrut}{\Tstrut\Bstrut} 

\definecolor{Gray}{gray}{0.85}
\definecolor{LightCyan}{rgb}{0.88,1,1}

\newcolumntype{a}{>{\columncolor{Gray}}c}
\newcolumntype{b}{>{\columncolor{Gray}}r}
\newcolumntype{d}{>{\columncolor{white}}c}

\usepackage{hyperref}
\hypersetup{
    colorlinks=true,
    citecolor=blue,
    linkcolor=blue,
    filecolor=cyan,      
    urlcolor=blue,
}
 
\providecommand{\keywords}[1]
{
  \small	
  \textbf{\textit{Keywords---}} #1
}

\title{
Convolutions and Mixtures of Gamma, Stable and Mittag-Leffler Distributions}
\author{Nomvelo Karabo Sibisi \\{\small {\tt sbsnom005@myuct.ac.za}}}
\date{\today}

\begin{document}
\maketitle
\thispagestyle{empty}


\begin{abstract}
\noindent
This paper uses convolutions  of  the   gamma  density and  the one-sided  stable density  
 to construct  higher level densities.
The approach is applied to constructing a 4-parameter  Mittag-Leffler density,
whose Laplace transform is a  corresponding  Mittag-Leffler  function,
which is completely monotone (CM) by construction.
Laplace  transforms of mixtures of the stable densities with respect to the 4-parameter Mittag-Leffler distribution 
are compositions of the  Mittag-Leffler  functions with   Bernstein functions,
thereby generating  a rich family of CM variants of the base CM Mittag-Leffler  functions, 
including known instances as special cases.
\end{abstract}
 

\section{Introduction}
\label{sec:introduction}

The   goal of this paper is to use 
convolutions  and mixtures involving   the familiar  gamma   and   one-sided  stable distributions   
 to construct  higher level distributions.
 We apply this approach to constructing  Mittag-Leffler distributions,
 whose Laplace-Stieltjes transforms are  Mittag-Leffler functions in various guises.
This is  a `forward problem' that involves constructing  distributions and  evaluating  their transforms,
 as opposed to  an `inverse problem' that seeks  to infer densities from given functions through complex analytic  inversion
 (Laplace or Stieltjes inversion, depending on the   transform linking the desired density to the function at hand).
 Such  inversion cannot, by itself, ensure that the  result will be nonnegative,
  supplementary conditions typically need to be imposed.
 Common though  inversion methods may be, our primary objective  strictly remains  one of direct construction  of distributions.
 In any event,  Mittag-Leffler distributions have applications in their own right to model, for example, random partitions of the integers
 or the growth of random trees and graphs.
  In such context, their property as  generators of  Mittag-Leffler functions through Laplace-Stieltjes transforms may only  be of passing interest.
 
It is well-known  that  the Mittag-Leffler function $E_\alpha(-x^\alpha)$  for  $0<\alpha<1$ is   the  Laplace-Stieltjes transform of a 
distribution with a straightforward  density  
 (equivalently,  $E_\alpha(-x^\alpha)$  is the ordinary Laplace transform of said  density). 
It is also   known 
that  $E_\alpha(-x)$  for  $0<\alpha<1$ is  the Laplace-Stieltjes transform 
of a  distribution known as the Mittag-Leffler distribution, 
which  has a known   density too. 
The primary question  is whether  the relationship  might also hold for some other  power of  of $x$ ($\sigma$, say),  
at least for the same restriction $0<\alpha<1$ suggested by the two examples cited.
To be precise: \\
a) Is $E_\alpha(-x^\sigma)$ for  certain  $\sigma\notin \{1,\alpha\}$ also the Laplace-Stieltjes transform of a  distribution? \\ 
b) If (a) holds, what is the explicit form of the distribution associated with $E_\alpha(-x^\sigma)$?

Part~(a) of the question   has a simple answer, which may be looked upon as arising from the  concept of completely monotone functions.
A function  is said to be completely monotone (CM) if  it  and all its derivatives of alternating sign are nonnegative.
By a theorem due to Bernstein, a function is CM if and only if it is  the Laplace-Stieltjes transform of a  distribution.
$E_\alpha(-x)$ and  $E_\alpha(-x^\alpha)$ are  thus examples of  CM functions.
Furthermore, by a property  of CM functions, if  $E_\alpha(-x)$ is CM so is the composition $E_\alpha(- \lambda x^\sigma)$  for $\lambda>0$ and
$0<\sigma<1$. 
Hence we  have the answer to (a): there exists a distribution whose  Laplace-Stieltjes transform is $E_\alpha(-x^\sigma)$ for 
 $0<\alpha<1$ and  $0<\sigma<1$.
 This clearly  includes the  case $\sigma=\alpha$,    {\it i.e.}\  $E_\alpha(-x^\alpha)$ is CM because $E_\alpha(-x)$ is CM.
It is also  true  that $E_\alpha(-(x^{\alpha})^\sigma)$ is CM because $E_\alpha(-x^\alpha)$ is CM, but this is narrower since 
we cannot deduce  that $E_\alpha(-x)$ is CM by using the CM nature of $E_\alpha(-x^\alpha)$ as a  starting point.
 
 Part~(b) seeks to go beyond the existence statement and determine the explicit distribution associated with $E_\alpha(-\lambda x^\sigma)$.
 This is core to what this paper  seeks to achieve:  explicit construction of distributions  that  mirror 
 corresponding  CM functions.
 In particular, since $E_\alpha(- \lambda x^\sigma)$ is CM by virtue of the fact that $E_\alpha(-x)$ is CM,
 we shall establish that the distribution whose  Laplace-Stieltjes transform is $E_\alpha(-\lambda x^\sigma)$ 
 arises from the  Mittag-Leffler distribution, whose  Laplace-Stieltjes transform is $E_\alpha(-x)$. 
 This is despite the fact that  the well-known case $E_\alpha(-\lambda x^\alpha)$ has been  studied directly and independently  of the
 $E_\alpha(-x)$ case in the literature.
 
  \subsection{Contribution}
\label{sec:contribution}

 In this paper, convolutions  and mixtures of gamma and stable distributions are the fundamental concepts  that underpin all higher distributions.
 As a starting point, we will use this methodology  to construct  the  Mittag-Leffler distribution, 
 which involves  one parameter $0<\alpha<1$ 
 (we can  include the ends of this interval with an extended definition of the stable density, but this is not a key conceptual  point here).
It  then becomes evident how to generalise to the construction  of multi-parameter Mittag-Leffler distributions.

For a   `quick read',  Tables~\ref{table:MLDensities},~\ref{table:ML3parDensities},~\ref{table:ML4parDensities},~{\ref{table:ML4parBetaML}   
summarise  the   results derived  in the paper,  including known instances as special cases.
However, the  methodology  presented in the paper, based on convolution and mixing  of basic distributions, is more broadly applicable than 
the illustrative examples shown in the Tables. 

\section{Background}
\label{background}

\subsection{Mittag-Leffler Functions}
\label{sec:MLfunction}

Gorenflo {\it et al.}~\cite{gorenflo2014mittag} is a comprehensive treatise on Mittag-Leffler functions and their applications.
The 3-parameter Mittag-Leffler function, also known as the Prabhakar function,
has an infinite series representation for $z\in \mathbb{C}$
\begin{align}
E^\gamma_{\alpha,\beta}(z) &= \frac{1}{\Gamma(\gamma)} 
   \sum_{k=0}^\infty \frac{\Gamma(\gamma+k)}{k!\,\Gamma(\alpha k+\beta)}\, z^k  \qquad {\rm Re}(\alpha)>0, \, {\rm Re}(\beta)>0, \, \gamma > 0
\label{eq:ML3parseries}
\end{align}
The termwise Laplace transform of $x^{\beta-1}E^\gamma_{\alpha,\beta}(-\lambda x^\alpha)$  $(\lambda>0)$ sums to
\begin{align}
\int_0^\infty e^{-sx} \, x^{\beta-1}E^\gamma_{\alpha,\beta}(-\lambda x^\alpha) \,dx  
  &= \frac{s^{\alpha\gamma-\beta}}{(\lambda+s^\alpha)^\gamma} 
\label{eq:LaplaceML3par}
\end{align}
As discussed in Gorenflo {\it et al.}~\cite{gorenflo2014mittag} and numerous other references such as the survey by Giusti {\it et al.}~\cite{Giusti}, 
there are  many other attributes of Mittag-Leffler functions, often related to
fractional  calculus and its applications, but these  will not be of direct interest here.

What is key  in this paper is that we shall often interpret  $E_\alpha(-y x^\alpha)$ 
as a function of two variables $x$ and $y$, not merely as  a function of $x$ for some fixed parameter $y$.
Thus, we shall   typically work with $E_\alpha(-y x^\alpha)$ as a  function of $x$ given $y$ or $y$ given $x$.

\subsection{Mittag-Leffler Distributions}
\label{sec:MLdistributions}

Pollard~\cite{PollardML}  used  complex analytic methods to derive the distribution  
$P_{\alpha}(t)$ $(t>0)$ for $0<\alpha<1$, with  Laplace-Stieltjes transform $E_\alpha(-x)$, $x\ge0$.
In a probabilistic  derivation, Feller~\cite[XIII.8]{Feller2} considered the two-dimensional Laplace  transform of the 
one-sided  $\alpha$-stable distribution with an associated  scale factor. 
$P_{\alpha}$ has  become known as the (1-parameter) Mittag-Leffler   distribution.

In the context of studying  random partitions of the integers,
Pitman~\cite[Theorem~3.8]{Pitman_CSP}  derived the 2-parameter Mittag-Leffler   distribution $P_{\alpha,\theta}$ 
where $\alpha$ is again  the index of the one-sided $\alpha$-stable  distribution $(0<\alpha<1)$
and  $\theta>-\alpha$ is associated with  ``polynomial tilting" of the $\alpha$-stable density.
$P_{\alpha,\theta}$  is also referred to as the generalized Mittag-Leffler distribution and alternatively 
denoted by ${\rm ML(\alpha,\theta})$ 
(Goldschmidt and Haas~\cite{GoldschmidtHaas}, Ho~{\it  et al.}~\cite{HoJamesLau}). 
The Mittag-Leffler distribution is the $\theta=0$ case $P_{\alpha}\equiv{\rm ML}(\alpha,0)$.
A derivation of ${\rm ML(\alpha,\theta})$   due to Janson~\cite{Janson} takes  the form of  a limiting distribution of  a  P{\' o}lya urn scheme
(also see Flajolet~{\it  et al.}~\cite{Flajolet}).
The same limiting distribution arises from the Chinese restaurant process (Pitman~\cite[3.1]{Pitman_CSP}), 
which is  intimately associated with random partitions and  the 
Poisson-Dirichlet process ${\rm PD(\alpha,\theta})$ of Pitman and Yor~\cite{PitmanYor}.
Pursuing the Chinese restaurant metaphor, M\"{o}hle~\cite{Mohle} added a cocktail bar  in order to construct a
3-parameter Mittag-Leffler distribution  ${\rm ML(\alpha,\beta,\gamma})$. 

Mittag-Leffler distributions have finite moments of all order, thereby uniquely determining   the distributions.
Hence the distributions   are  often described by specification of their moments.
By studying such moments, M\"{o}hle~\cite[Corollary~3]{Mohle} proved that if $X$ has the 3-parameter Mittag-Leffler distribution,
then $X$ is equal in distribution  to the product  $YZ$ of independent random variables $Y,Z$,
where $Y$ is beta distributed and $Z$ has the 2-parameter Mittag-Leffler distribution
(also see Banderier~{\it  et al.}~\cite{Banderier}). 
In fact,    Bertoin and Yor~\cite[Lemma~6]{BertoinYor} gave an earlier proof of the identity 
without  explicit mention of  the  Mittag-Leffler distribution beyond stating its moments.
 Explicit  mention was  subsequently made   by  Ho~{\it  et al.}~\cite[Proposition ~2.1,~Remark~2.3]{HoJamesLau}.

In another  context, Mittag-Leffler distributions  play a fundamental role in the probabilistic modelling of the evolution of  random trees and graphs, 
often used as models of real-world network behaviour. 
In the study of stable trees (Goldschmidt and Haas~\cite{GoldschmidtHaas}, Rembart and Winkel~\cite{RembartWinkel2018, RembartWinkel2023})
and stable graphs (Goldschmidt~{\it  et al.}~\cite{ Goldschmidt22}), 
the distributions that feature as   building  blocks are beta, generalised
Mittag-Leffler ${\rm ML(\alpha,\theta})$, Dirichlet and Poisson-Dirichlet  ${\rm PD(\alpha,\theta})$.
James~\cite{James2015}, S\'{e}nizergues~\cite{Senizergues} discussed the appearance of  
${\rm ML(\alpha,\theta})$ as limit distributions  in the   growth of random graphs
 inspired by  the preferential  attachment model due to Barab\'{a}si and  Albert~\cite{BarabasiAlbert}.
 ${\rm ML(\alpha,\theta})$ also arises in a random walk model of  network formation by Bloem-Reddy and Orbanz~\cite{BloemReddyOrbanz}.
 
In yet another  context, the complete memory  elephant random walk with  delays converges to the Mittag-Leffler distribution
 (Bercu~\cite{Bercu2022}, Janson~\cite[Example~14.10]{Janson2024}, Miyazaki and Takei~\cite{MiyazakiTakei}).
  
 More generally, Mittag-Leffler distributions  belong to a class that Janson~\cite{JansonProbSurv} referred to as
 distributions with ``moments of Gamma type''.
In an addendum, Janson~\cite{JansonProbSurvA}  added that the densities of this class were $H$ functions 
(Fox~\cite{Fox}, Mathai {\it  et al.}~\cite{Mathai}), with  Meijer $G$ functions as a subclass.
The distributions are  often described  in terms of their moments, thereby facilitating the use of  
``Mellin transform techniques to obtain expansions or asymptotics of the density function"~\cite{JansonProbSurv}.
Pursuing a similar theme, Beghin {\it  et al.}~\cite{BeghinFoxH} explored Fox-$H$  densities with moments of all order,
whose Laplace transforms are generalised Wright functions. 
Under a particular  choice of parameters, Beghin {\it  et al.}~\cite[Example~5.5]{BeghinFoxH} 
 identified a Fox-$H$ density whose Laplace transform is 
$\Gamma(\beta)E^\gamma_{\alpha,\beta}(-x)$ for $0<\alpha<1$ and $ \beta-\alpha\gamma=1-\alpha$.

 It thus seems  worthwhile to study Mittag-Leffler distributions in  their own right -- in the case of this paper, from a 
 perspective of convolutions and mixtures of distributions rather than as limiting distributions of random selection processes.
 In addition, the paper explores in some detail  the duality, mediated by the Laplace-Stieltjes transform, 
 between  Mittag-Leffler distributions and corresponding  Mittag-Leffler functions in  various guises.
Notably, the Laplace-Stieltjes transform  of  ${\rm ML(\alpha,\theta})$ 
is $\Gamma(1+\theta)E^{1+\theta/\alpha}_{\alpha,1+\theta}(-x)$ (Ho {\it  et al.}~\cite{HoJamesLau},  James~\cite{James_Lamperti}).

A key highlight of this paper is the derivation of Mittag-Leffler densities using  straightforward  properties of convolutions, 
without any need to appeal to 
complex analytic inversion  or  special function representation.

\subsubsection{4-parameter Mittag-Leffler Distribution}
\label{sec:4par}

Pursuing the approach  of mixtures  and convolutions of stable and gamma distributions, 
this paper derives a 4-parameter Mittag-Leffler distribution  
$P_{\alpha,\beta,\gamma,\theta}\equiv{\rm ML(\alpha,\beta,\gamma,\theta})$
for $0<\alpha<1, \beta>\alpha\gamma,   \theta>-\alpha\gamma$,
whose Laplace-Stieltjes transform  is $\Gamma(\beta+\theta)E^{\gamma+\theta/\alpha}_{\alpha,\beta+\theta}(-x)$ $(x\ge0)$.
The difference $\beta-\alpha\gamma>0$ plays an important  role.
As we shall demonstrate, for $\beta-\alpha\gamma=1-\alpha$ the  convolution essentially  simplifies to the  $\alpha$-stable density  
(Proposition~\ref{prop:stable} below gives the precise statement).
${\rm ML(\alpha,\theta})\equiv {\rm ML(\alpha,1,1,\theta})$ is one   instance of $\beta-\alpha\gamma=1-\alpha$.
 
With the parameter restrictions stated, $E^{\gamma+\theta/\alpha}_{\alpha,\beta+\theta}(-x)$ is CM.
Hence 
$E^{\gamma+\theta/\alpha}_{\alpha,\beta+\theta}(-\lambda x^\sigma)$ $(\lambda>0, 0<\sigma<1)$ is also CM.
We derive the explicit  form of the associated  density,
a mixture of a $\sigma$-stable density with respect to  ${\rm ML}(\alpha,\beta,\gamma,\theta)$,
with Laplace transform $\Gamma(\beta+\theta)E^{\gamma+\theta/\alpha}_{\alpha,\beta+\theta}(-\lambda x^\sigma)$.
Through multiplication with other CM functions and 
iterative composition, we may generate a cascade of  CM functions
as Laplace transforms of associated explicit densities. 
We discuss the general construction principle but demonstrate  only a limited selection  of such possibilities. 

\subsection{Other Perspectives}
\label{sec:other}

\subsubsection{Other Definition of Mittag-Leffler Distributions}
\label{sec:otherMLdistributions}
Pillai~\cite{Pillai} used the term `Mittag-Leffler distribution' to refer to $1-E_\alpha(-\lambda x^\alpha)$ $(\lambda>0, 0<\alpha\le1)$
(Pillai defined it with $\lambda=1$),
with Laplace-Stieltjes transform 
\begin{align}
\int_0^\infty e^{-sx} \, d(1-E_\alpha(-\lambda x^\alpha)) &= 1- \frac{s^\alpha}{\lambda+s^\alpha}  = \frac{\lambda}{\lambda+s^\alpha} 
\label{eq:LStPillai}
\end{align}
In this context, `generalized Mittag-Leffler distribution' is taken to refer to the distribution with  Laplace-Stieltjes transform
$\lambda^\gamma/(\lambda+s^\alpha)^\gamma$, $\gamma>0$ 
(Fujii~\cite{Fujii}, Haubold {\it  et al.}~\cite{Haubold}, Jose {\it  et al.}~\cite{Jose}, 
 Korolev {\it  et al.}~\cite{Korolev2020}, Korolev and Zeifman}~\cite{Korolev2016})).
Pakes~\cite{Pakes95}  helpfully referred to  this distribution
as   the  generalized positive Linnik distribution that we may denote by ${\rm PL}(\alpha,\gamma,\lambda)$  
(also see Barabesi {\it  et al.}~\cite{Barabesi}).
For $\gamma=1$, ${\rm PL}(\alpha,\lambda)\equiv {\rm PL}(\alpha,1,\lambda)$, 
explicitly ${\rm PL}(x\vert \alpha,\lambda)=1-E_\alpha(-\lambda x^\alpha)$.
Despite this,  the use of `Mittag-Leffler distributions'  for the positive Linnik distributions persists in the literature.
 Pakes~\cite{Pakes95}  and  several other authors have commented on this ``unfortunate attribution", 
which  causes avoidable  confusion with the Mittag-Leffler distributions of~\ref{sec:MLdistributions}.

We shall reserve the term  `Mittag-Leffler distributions' exclusively  for the distributions discussed  in \ref{sec:MLdistributions} above.
 For completeness, we shall comment on   how the positive Linnik distribution arises in the context of our discussion.

\subsubsection{Complex Analytic  Inversion}
\label{sec:otherComplex}

There is a body of literature that places emphasis on  exploring  the CM character  of Mittag-Leffler  functions, with the fact that they are   Laplace transforms
 of nonnegative functions  seeming to arise   as a means to that end  rather than as the primary objective.
 Such literature might thus be described as more  analytic  than probabilistic. 

Pollard~\cite{PollardML}  took the complex analytic  Laplace inverse of the 1-parameter instance of (\ref{eq:LaplaceML3par}) as the starting point 
to establish the CM character of  $E_\alpha(-x)$.
Inspired  by this, 
G\'{o}rska {\it  et al.}~\cite{Gorska} proved that $E^\gamma_{\alpha,\beta}(-x)$ is CM, 
{\it i.e.} that it is  the Laplace  transform of a  nonnegative function 
for  $0<\alpha\le1,\beta>\alpha\gamma,\gamma>0$ 
(also see  Tomovski {\it et al.}~\cite[Theorem~2]{Tomovski}).
 G\'{o}rska {\it  et al.}~ \cite[Remark~2.2]{Gorska}  noted that
 the CM nature of  $E^\gamma_{\alpha,\beta}(-x)$  implies that: 
(a) $E^\gamma_{\alpha,\beta}(-x^\alpha)$ is CM by the composition property of CM functions and 
(b) $x^{\beta-1}E^\gamma_{\alpha,\beta}(-x^\alpha)$ is CM by the  product property, with the additional  constraint $\beta\le1$.

The CM character  of (b) had previously been investigated  by 
Capelas de Oliveira {\it et al.}~\cite{Oliveira} using  Stieltjes inversion of the  Laplace  transform  (\ref{eq:LaplaceML3par})
of  $x^{\beta-1}E^\gamma_{\alpha,\beta}(-x^\alpha)$.
  Stieltjes inversion by itself 
  (Titchmarsh~\cite[11.8 (p318)]{Titchmarsh}, 
Widder~\cite[VIII.7 (p341)]{Widder}) 
 does not   shed light on whether the inverse is a nonnegative function,
 the paper determined the required parameter restrictions 
 $0<\alpha\le1$, $0<\alpha\gamma\le\beta\le1$
 using a theorem  due to Gripenberg~\cite{Gripenberg}. 
Mainardi and Garrappa~\cite{MainardiGarrappa}  pursued  the same inversion approach, but with a variation  on the argument 
leading to the same parameter restrictions.

The condition $\beta\le1$  arises from treating $x^{\beta-1}$ and $E^\gamma_{\alpha,\beta}(-x^\alpha)$ 
as separate  CM functions, so that the product is also CM.
Amongst other things, we will show that $\beta\le1$ is, in fact, redundant for  $x^{\beta-1}E^\gamma_{\alpha,\beta}(-x^\alpha)$ 
to be the Laplace transform of a density, {\it i.e.}\ for it to be CM.

\section{Preliminaries}
\label{sec:preliminaries}

Let $\psi(x),\, x>0$ be the Laplace-Stieltjes transform of a distribution function $F(t)$, $t>0$
\begin{align}
\psi(x) &= \int_0^\infty e^{-xt} \, dF(t)
\label{eq:LStransform}
\end{align}
If $dF(t)=f(t)dt$ for a density $f(t)$, 
 then $\psi(x)$ is equivalently the (ordinary) Laplace transform of $f(t)$.
$\psi(x)$ of the form~(\ref{eq:LStransform}) is said to be completely monotone (CM), 
{\it i.e.} its  derivatives of alternating sign are nonnegative: $(-1)^n\psi^{(n)}(x)\ge0, n\ge0$.
In fact, by a theorem due to Bernstein (Feller~\cite[XIII.4]{Feller2}, Widder~\cite[p160]{Widder}),  
a function is CM if and only if  it is of the form~(\ref{eq:LStransform}).
The topic of completely monotone functions  is well-studied 
({\it e.g.}\ Feller~\cite[XIII.4]{Feller2}, Schilling~{\it et al.}~\cite{Schilling}, Widder~\cite[IV]{Widder}).
Properties of   interest here are:

\begin{proposition}[Product]
If $\psi_1(x)$ and  $\psi_2(x)$ are {\rm CM}, so is their product $\psi_1(x)\psi_2(x)$.

\label{prop:product}
\end{proposition}
\begin{proposition}[Composition]
If $\psi(x)$ is {\rm CM},  so is the composition $\psi(\eta(x))$ where $\eta(x)$ is a nonnegative function with a {\rm CM} derivative, 
also known as a Bernstein function.
\label{prop:composition}
\end{proposition}

We study properties of  distributions corresponding to  the properties above for  CM functions.
The product case is straightforward. 
If the CM functions $\psi_1(x), \psi_2(x)$ are Laplace transforms of densities $f_1(t),f_2(t)$ respectively,
then the convolution theorem states that the product $\psi_1(x)\psi_2(x)$ is  the Laplace transform of  
the convolution density $\{f_1 \star f_2\}(t)$ 
\begin{align}
\{f_1 \star f_2\}(t) &= \int_0^t f_1(t-u)f_2(u)\, du 
= \int_0^t f_2(t-u)f_1(u)\, du 
\label{eq:conv}
\end{align}

The composition case is more subtle.
If  the CM function $\psi(x)$ is the Laplace-Stieltes transform of a  known  distribution, what can be said about
 the   distribution  whose   Laplace-Stieltes transform is 
$\psi(\eta(x))$ for a Bernstein function $\eta(x)$? 
Rather than address the question for the very broad class of Bernstein functions  $\eta(x)$, we  consider the case
$\lambda x^\alpha$ for  $\lambda>0$ and $0<\alpha<1$.

\subsection{Gamma  Distribution}
 \label{sec:gamma}
 
 The   gamma distribution $G(\gamma,\lambda)$ with density  $\rho_{\gamma,\lambda}$
 with shape   and scale parameters $\gamma, \lambda>0$  is
\begin{align}
dG(t \vert \gamma,\lambda) 
&\equiv  \rho_{\gamma,\lambda}(t)  \, dt  
= \frac{\lambda^\gamma}{\Gamma(\gamma)}\,t^{\gamma-1} \,e^{-\lambda t} dt 
\label{eq:gamma} 
\end{align}
We shall make frequent reference to the unnormalised power density $\rho_\gamma(t) =  t^{\gamma-1}/\Gamma(\gamma)$.
The   Laplace-Stieltjes  transform of $G(\gamma,\lambda)$ (or the  Laplace transform of its density $\rho_{\gamma,\lambda}$)
 and the Laplace transform  of $\rho_\gamma(t)$  are respectively
\begin{alignat}{3}
\int_0^\infty e^{-xt}\, dG(t \vert \gamma,\lambda) 
\equiv \int_0^\infty e^{-xt}\,\rho_{\gamma,\lambda}(t) \, dt &= \frac{\lambda^\gamma}{(\lambda+x)^\gamma} &&\quad (x\ge0)
\label{eq:gammaLT}  \\
\int_0^\infty e^{-xt}\, \rho_\gamma(t)\, dt &=  x^{-\gamma} &&\quad (x>0)
\label{eq:powerLT}
\end{alignat}  
We also  let  $\rho_{\gamma=0}(t)\equiv\delta(t)$ with Laplace transform~1 ($\delta(t)$ is the Dirac delta function).

 
\subsection{Beta  Distribution}
 \label{sec:beta}
 
 The beta distribution ${\rm Beta}(\alpha,\beta)$, for given parameters $\alpha,\beta>0$, is defined on $(0,1)$ with density 
 \begin{align}
 {\rm beta}(t\vert \alpha,\beta) &= \frac{\Gamma(\alpha+\beta)}{\Gamma(\alpha)\Gamma(\beta)} \, t^{\alpha-1}(1-t)^{\beta-1} 
\label{eq:beta} \\
\int_0^1 t^k \, {\rm beta}(t\vert \alpha,\beta)\, dt  &= \frac{\Gamma(\alpha+\beta)\Gamma(\alpha+k)}{\Gamma(\alpha)\Gamma(\alpha+\beta+k)}
\quad ({\rm moments}
 \label{eq:betamoments})
\end{align} 

In anticipation of later discussion,  let  $U,V>0$ be independent  random variables whose  distributions  have densities $f_U(u), f_V(v)$ respectively.
Then it is a standard result that  the product $T=VU$ has density $f_T(t)$ given by 
\begin{align}
f_T(t) &= \int_0^\infty \int_0^\infty \delta(t-uv) f_U(u) f_V(v) \,du\, dv
      = \int_0^\infty f_V(t/u)  f_U(u) \frac{du}{u} 
\label{eq:UVdensity}
\end{align}
In particular, if $ f_U(u)\equiv {\rm beta}(u\vert \alpha,\beta)$ on $(0,1)$ and 0 on $u\ge1$, then 
\begin{align}
f_T(t)  &= \int_0^1 f_V(t/u) \, {\rm beta}(u\vert \alpha,\beta) \, \frac{du}{u}  \nonumber \\
 &=  \frac{\Gamma(\alpha+\beta)}{\Gamma(\alpha)\Gamma(\beta)} \int_0^1 f_V(t/u) \, u^{\alpha-1}(1-u)^{\beta-1}\, \frac{du}{u} 
 \label{eq:betaVdensity}
\end{align}

 \subsection{Stable Distribution}
\label{sec:stable}

The one-sided  stable distribution $F_\alpha(t\vert \lambda), t>0$ ($0<\alpha<1$), 
conditioned  on a  scale parameter $\lambda>0$, is indirectly defined by   its Laplace-Stieltjes transform 
$($equivalently, the ordinary Laplace transform of its density $f_\alpha(t\vert  \lambda)$)
\begin{align}
e^{-\lambda x^\alpha} &= \int_0^\infty e^{-xt}\,dF_\alpha(t\vert \lambda) =  \int_0^\infty e^{-x t}  f_\alpha(t\vert \lambda)\, dt
\label{eq:stable} 
\end{align}
$f_\alpha(t\vert \lambda)$ may also be written as $ f_\alpha(t \lambda^{-1/\alpha})\, \lambda^{-1/\alpha}$, 
where $f_\alpha(t)\equiv f_\alpha(t\vert \lambda=1)$.

$e^{-\lambda x^\alpha} $ is a CM function by construction.
In particular, minus its first derivative is 
\begin{align}
\lambda\alpha x^{\alpha-1} e^{-\lambda x^\alpha} &= 
\int_0^\infty e^{-x t}  tf_\alpha(t\vert \lambda)\, dt
= \lambda\alpha \int_0^\infty e^{-x t} \{\rho_{1-\alpha}\star f_\alpha(\cdot\vert\lambda)\}(t)\, dt
\label{eq:stableLTderiv}
\end{align}
where $\rho_{1-\alpha}(t)=t^{-\alpha}/\Gamma(1-\alpha)$ is the power density $\rho_\gamma(t)$ of Section~\ref{sec:gamma}  for $\gamma=1-\alpha$.
The second equality arises because   $x^{\alpha-1} e^{-\lambda x^\alpha}$ is the Laplace transform of  the convolution 
$\{\rho_{1-\alpha}\star f_\alpha(\cdot\vert\lambda\}(t)$.  
Uniqueness of Laplace transforms  implies the following known property of the stable density
\begin{proposition}
\label{prop:stable}
Let $f_\alpha(t\vert \lambda)$ $(0<\alpha<1, t>0)$ be the stable density with scale factor $\lambda>0$.
Then the product $tf_\alpha(t\vert \lambda)$  is proportional to the  convolution  $\{\rho_{1-\alpha}\star f_\alpha(\cdot\vert\lambda)\}(t)$ of 
the power density  $\rho_{1-\alpha}(t)=t^{-\alpha}/\Gamma(1-\alpha)$ and $ f_\alpha(t\vert \lambda)$
\begin{align}
tf_\alpha(t\vert \lambda) &\equiv \lambda\alpha \{\rho_{1-\alpha}\star f_\alpha(\cdot\vert\lambda)\}(t)
= \frac{\lambda\alpha}{\Gamma(1-\alpha)} \int_0^t (t-u)^{-\alpha} f_\alpha(u \vert\lambda) \, du
\label{eq:stableLTderivdensity} 
\end{align}
\end{proposition}
We will refer to  Proposition~\ref{prop:stable} in  discussion of the Mittag-Leffler distribution.

Proposition~\ref{prop:stable} may  be viewed as a particular instance of the representation of
infinitely divisible  distributions on the nonnegative half-line $\mathbb{R}_{+}=[0,\infty)$, a class that includes the stable distribution,
Feller~\cite[XIII.7]{Feller2}, Steutel and van Harn~\cite[III (Theorem~4.17)]{SteutelvanHarn}
(also see Sato~\cite{Sato}).  
There is also an intimate  link to the generalised gamma convolutions studied by Bondesson~\cite{Bondesson}.

We note that  $\{\rho_{1-\alpha}\star f_\alpha(\cdot\vert\lambda\}(t)$ may be regarded as the Riemann-Liouville fractional integral 
$\{I_+^{1-\alpha}\, f_\alpha(\cdot\vert\lambda\}(t)$ of the stable density $f_\alpha(\cdot\vert\lambda)$,
where $\{I_+^\nu\, f\}(t)$ for a function $f(t)$ ($t\ge0, \nu>0$) is 
the convolution $\{\rho_\nu\star f\}(t)$,  $\rho_{\nu}(t)=t^{\nu-1}/\Gamma(\nu)$
(Kiryakova~\cite{kiryakova1993generalized}, Miller and Ross~\cite{miller1993introduction}, Samko~{\it et al.}~{\cite{samko1993fractional}).
We  shall continue to use the general language of convolutions than that of fractional integrals.
We further note 
that Proposition~\ref{prop:stable} 
has  inspired the study of generalised stable densities $f_{\alpha,m}(t)$ 
defined by  $t^m f_{\alpha,m}(t) = \alpha\{\rho_{1-\alpha}\star  f_{\alpha,m}\}(t)$ for $m>0$
(Jedidi~{\it et al.}~\cite{Jedidi},  Pakes~\cite{Pakes}, Schneider~\cite{Schneider}).
We shall not have cause for such generalisation here.

The stable density $f_\alpha(t\vert \lambda)$  is only known in closed form for selected values of $\alpha$, the simplest  being for $\alpha=1/2$.
For the  general  rational  case $\alpha=l/k$ for positive integers $l,k$ with $k>l$, Penson and G\'{o}rska~\cite{PensonGorska} expressed
 $f_\alpha$ in terms of the Meijer $G$ function
 (also see James~\cite[8]{James_Lamperti}).

 Pollard~\cite{Pollard} (also see Feller~\cite[XVII.6]{Feller2}) proved that $f_\alpha(t\vert\lambda)$ 
has the   infinite series representation
 (with  modification here to include conditioning on the scale factor $\lambda>0$)
\begin{align}
  f_\alpha(t\vert \lambda)  &= \phantom{-} \frac{1}{\pi}\,{\rm Im}  \int_0^\infty e^{-tu} \, e^{-\lambda (e^{-i\pi}u)^\alpha}  \, du  
\label{eq:Pollard} \\
&= \phantom{-} \frac{1}{\pi}\,{\rm Im}
   \sum_{k=0}^\infty \frac{(-\lambda)^k}{k!} e^{-i\pi\alpha k}  \int_0^\infty e^{-tu} \, u^{\alpha k} \, du
\label{eq:Pollardinfsum1} \\
&= -\frac{1}{\pi}
   \sum_{k=0}^\infty \frac{(-\lambda)^k}{k!} \sin(\pi\alpha k) \, \frac{\Gamma(\alpha k+1)}{t^{\alpha k+1}}
\label{eq:Pollardinfsum2} 
\end{align}
We may equivalently sum from $k=1$ since the $k=0$ term is identically 0.

\section{Main Contribution}
\label{sec:main}

\begin{theorem}
\label{thm:CompositionLaplace}
Let  $\psi(x), x\ge0$ be the Laplace-Stieltjes transform of a   distribution  $F(t)$, $t>0$
\begin{align}
\psi(x) &= \int_0^\infty e^{-x t}\,dF(t) \quad x\ge0
\label{eq:LSF} 
\end{align}
Then, for $0<\alpha< 1$ and $\lambda>0$, the composition $\psi(\lambda x^\alpha)$ 
 is the  Laplace transform of a  mixture density $w_\alpha(t\vert \lambda)$ 
\begin{align}
\psi(\lambda x^\alpha) 
&=  \int_0^\infty e^{-x t} \, w_\alpha(t\vert \lambda)\, dt 
\label{eq:LTcomposition} \\
\textrm{where}\quad w_\alpha(t\vert \lambda) &=  \int_0^\infty   f_\alpha(t\vert \lambda u) \,dF(u)
\label{eq:w}
\end{align}
$f_\alpha(t\vert \lambda u)$ being the $\alpha$-stable density  conditioned on the scale factor $\lambda u$.
\end{theorem}
\begin{proof}[Proof of Theorem  \ref{thm:CompositionLaplace}]
\label{proof:CompositionLaplace}
$e^{-\lambda u x^\alpha}$ ($0<\alpha< 1$) is 
the Laplace transform of the stable density $f_\alpha(\cdot\vert \lambda u)$. 
Hence
\begin{align*}
\psi(\lambda x^\alpha) = \int_0^\infty e^{-\lambda x^\alpha u}\,dF(u) 
&= \int_0^\infty  \int_0^\infty e^{-x t}  f_\alpha(t\vert \lambda u) dt \,dF(u)  \\
&= \int_0^\infty e^{-x t} \left[ \int_0^\infty   f_\alpha(t\vert \lambda u)\,dF(u) \right] \, dt \\
&= \int_0^\infty e^{-x t} \, w_\alpha(t\vert \lambda) \, dt
\end{align*}
where  $w_\alpha(t\vert \lambda)$ is the integral in square braces, thereby proving~(\ref{eq:w}).
\end{proof}

For clarity, we  may  explicitly condition on  $F$ or its density $f$: 
{\it e.g.}\ $w_\alpha(t\vert \lambda)\to w_\alpha(t\vert \lambda,F)\equiv w_\alpha(t\vert \lambda,f)$.
Equivalently, we may  condition on any parameters that are known  to characterise $F$ uniquely,
{\it e.g.}\  $w_\alpha(t\vert \lambda, \gamma,y)$ 
where $\gamma,y$ are understood to be the shape and scale parameters of the gamma distribution $G(\gamma,y)$.

Theorem~\ref{thm:CompositionLaplace} states a known property that can be couched in the language  of subordination, 
as described by Feller~\cite[XIII.7]{Feller2}.
Of more direct interest here  are  the associated corollaries below.

\begin{corollary}
\label{cor:CompositionLaplaceconv}
Let  $\rho_\nu(t)=t^{\nu-1}/\Gamma(\nu)$ $(\nu>0)$ 
 and let $f(t)$  be the density of $F(t)$. 
 Then the following holds
\begin{align}
\{\rho_{\nu} \star w_\alpha(\cdot\vert \lambda,F)\}(t) 
&\equiv \int_0^\infty \{\rho_{\nu} \star f_\alpha(\cdot\vert \lambda u)\}(t)\, dF(u) \nonumber \\
&=  \lambda^{\nu/\alpha} \, w_\alpha(t\vert \lambda, \rho_{\nu/\alpha} \star f)
\label{eq:convrhow} \\
{\rm where}\quad
w_\alpha(t\vert \lambda, \rho_{\nu/\alpha} \star f) &=  \int_0^\infty f_\alpha(t\vert \lambda u)\, \{\rho_{\nu/\alpha} \star f\}(u) du \nonumber
\end{align}
\end{corollary}
 \begin{proof}[Proof of Corollary~\ref{cor:CompositionLaplaceconv}]
By the convolution  theorem, the Laplace  transform of the convolution  density $\{\rho_{\nu} \star w_\alpha(\cdot\vert \lambda,F)\}(t)$ is
the product $x^{-\nu}\psi(\lambda x^\alpha\vert F) =  \lambda^{\nu/\alpha} \left[(\lambda x^\alpha)^{-\nu/\alpha}\psi(\lambda x^\alpha\vert F)\right]$.
The product in square braces  is also  the  Laplace  transform of $w_\alpha(t\vert \lambda, \{\rho_{\nu/\alpha} \star f\})$.
Hence the validity of~(\ref{eq:convrhow}) by uniqueness of Laplace transforms.
 \end{proof}
 The point of Corollary~\ref{cor:CompositionLaplaceconv} is that it gives two different ways to express the
 convolution of the power density with the mixture density $w_\alpha(t\vert\lambda,F)$ --
 one involving  convolution  with the stable density and the other  involving  convolution with the density of $F$.
 This will have a bearing on our approach to the derivation of the Mittag-Leffler distribution below.

\begin{corollary}
\label{cor:CompositionLaplacePollard}
The mixture density $w_\alpha(t\vert \lambda,F)$ admits the representation 
\begin{align}
 w_\alpha(t\vert \lambda,F)  &= \frac{1}{\pi}\,{\rm Im}  \int_0^\infty e^{-tu} \,  \psi(\lambda e^{-i\pi\alpha}u^\alpha\vert F) \, du
 \label{eq:CompositionLaplacePollard}
 \intertext{In particular, for the convolution  $\rho_\nu\star f$ $(\nu>0)$}
 w_\alpha(t\vert \lambda, \rho_\nu\star f) 
  &= \frac{1}{\pi}\,{\rm Im} \int_0^\infty e^{-tu}\, (\lambda e^{-i\pi\alpha}u^\alpha)^{-\nu} \psi(\lambda e^{-i\pi\alpha}u^\alpha\vert f) \, du
 \label{eq:CompositionLaplacePollardConv}
\end{align}
\end{corollary} 
\begin{proof}[Proof of Corollary~\ref{cor:CompositionLaplacePollard}]
Using the Pollard representation~(\ref{eq:Pollard}) for $f_\alpha(t\vert \lambda u)$ gives 
\begin{align*}
 w_\alpha(t\vert \lambda,F) &=  \int_0^\infty   f_\alpha(t\vert \lambda u) \,dF(u) \nonumber \\
 &= \frac{1}{\pi}\,{\rm Im}  \int_0^\infty \int_0^\infty e^{-tv} \, e^{-\lambda u (e^{-i\pi}v)^\alpha} \, dv \, dF(u) \\
 &= \frac{1}{\pi}\,{\rm Im}  \int_0^\infty e^{-tv}   \int_0^\infty \, e^{-\lambda u (e^{-i\pi}v)^\alpha} \, dF(u)\, dv \\
 &= \frac{1}{\pi}\,{\rm Im}  \int_0^\infty e^{-tv}   \psi(\lambda e^{-i\pi\alpha}v^\alpha\vert F) \, dv
\end{align*}
thereby proving~(\ref{eq:CompositionLaplacePollard}).  
The convolution theorem for $dF(u)=\{\rho_\nu\star f\}(u)du$ leads to (\ref{eq:CompositionLaplacePollardConv}).
\end{proof}

\begin{corollary}
\label{cor:CompositionLaplaceInfseries}
 Let $F(t)$ 
 have finite moments $\{M_k\}$ of all order
\begin{align}
  M_k &= \int_0^\infty  u^k \,dF(u)  \equiv (-1)^k \psi^{(k)}(0)   < \infty \qquad k\ge0
  \label{eq:momentsF}
\end{align}
Then the mixture density $w_\alpha(t\vert \lambda)$  has the infinite series representation 
\begin{align}
w_\alpha(t\vert \lambda)
 &= \int_0^\infty   f_\alpha(t\vert \lambda u) \, dF(u) \nonumber \\
 &= -\frac{1}{\pi} \sum_{k=0}^\infty \frac{(-\lambda)^k}{k!} \sin(\pi\alpha k) \,  M_k \int_0^\infty e^{-tv} \, v^{\alpha k} \, dv
\label{eq:wPollardinfsum1} \\
 &= -\frac{1}{\pi} \sum_{k=0}^\infty \frac{(-\lambda)^k}{k!} \sin(\pi\alpha k) \, M_k \frac{\Gamma(\alpha k+1)}{t^{\alpha k+1}} 
\label{eq:wPollardinfsum2} 
\end{align}
\end{corollary}

\begin{proof}[Proof of Corollary~\ref{cor:CompositionLaplaceInfseries}]
Using the infinite series  representation~(\ref{eq:Pollardinfsum1}) of the stable density $f_\alpha(t\vert \lambda u)$ in~(\ref{eq:w}) gives
\begin{align*}
w_\alpha(t\vert \lambda)
  &= -\frac{1}{\pi} \int_0^\infty  \left[\sum_{k=0}^\infty \frac{(-\lambda u)^k}{k!} \sin(\pi\alpha k)\int_0^\infty e^{-tv} \, v^{\alpha k} \, dv \right]dF(u) \\
  &= -\frac{1}{\pi} \sum_{k=0}^\infty \frac{(-\lambda)^k}{k!} \sin(\pi\alpha k) \, \left[\int_0^\infty e^{-tv} \, v^{\alpha k} \, dv \right]
  \left[ \int_0^\infty u^k \, dF(u)\right]
\end{align*}
The rightmost term is $M_k$, thereby proving~(\ref{eq:wPollardinfsum1}) and hence (\ref{eq:wPollardinfsum2}).
 \end{proof}
 
We may now turn to a simple step function as our initial  choice of   $F$.


 \section{Unit Step   $F$}
 \label{sec:stepF}

 Let $F(t)$  be 0 on $0\le t<1$ with a jump of size 1 at $t=1$, {\it i.e.} $F(t)= \mathbbm{1}_{[1,\infty)}$,  
 where the indicator function  $\mathbbm{1}_{[c,\infty)}$ is 0 on $[0,c)$ and 1 on $[c,\infty)$.
The  density of $F(t)$  is $f(t)=\delta(t-1)$. 
The   Laplace-Stieltjes  transform of $F(t)$ (Laplace transform of $f(t)$) is 
\begin{align}
\psi(x) &= \int_0^\infty e^{-x t}\,dF(t)  
=  e^{-x} 
\implies \psi(\lambda x^\alpha)  = e^{-\lambda x^\alpha} \; (\lambda>0)
\label{eq:LSstepF}  
\end{align}
For $0<\alpha<1$, Theorem~\ref{thm:CompositionLaplace} states that $\psi(\lambda x^\alpha)$ 
is the Laplace transform of  a mixture density
\begin{align}
 \psi(\lambda x^\alpha)  &=  \int_0^\infty e^{-x t} \, w_\alpha(t\vert \lambda)\, dt = e^{-\lambda  x^\alpha} \quad 0<\alpha<1 
\label{eq:stepLTcomposition} \\ 
\textrm{where}\quad w_\alpha(t\vert \lambda) &=  \int_0^\infty   f_\alpha(t\vert \lambda u) \,dF(u) = f_\alpha(t\vert \lambda )
\label{eq:wstep}
\end{align}
For $\rho_\nu(t)=t^{\nu-1}/\Gamma(\nu)$ $(\nu>0)$, the convolution density $\{\rho_\nu\star f\}(t)$  has the  explicit form
 \begin{align}
 \{\rho_\nu\star f\}(t) &= \int_0^t \rho_\nu(t-u)\, dF(u) = \rho_\nu(t-1)\mathbbm{1}_{[1,\infty)}
 \label{eq:rhoconvdelta}
 \end{align}
 with Laplace transform  $x^{-\nu}\psi(x) = x^{-\nu}e^{-x}$.
\begin{theorem}
\label{thm:rhoconvstable}
For $0<\alpha<1$, $\rho_\nu(t)=t^{\nu-1}/\Gamma(\nu)$ $(\nu>0)$,  $\rho_0(t)=\delta(t)$ and $w_\alpha(t\vert \lambda)$ as defined in~$(\ref{eq:wstep})$,
the convolution density $\{\rho_\nu\star w_\alpha(\cdot\vert \lambda)\}(t)\equiv \{\rho_\nu\star f_\alpha(\cdot\vert \lambda) \}(t)$ $(\nu>0)$
 has the following  equivalent forms
\begin{align}
 \{\rho_\nu\star f_\alpha(\cdot\vert \lambda) \}(t) 
 &=  \lambda^{\nu/\alpha}  \int_0^\infty f_\alpha(t\vert \lambda u)\, \{\rho_{\nu/\alpha} \star f\}(u) du 
 \label{eq:rhoconvstable1} \\
 &=  \lambda^{\nu/\alpha}  \int_0^\infty f_\alpha(t\vert \lambda u) \rho_{\nu/\alpha}(u-1) \mathbbm{1}_{[1,\infty)} \, du  
 \label{eq:rhoconvstable2} \\
&= \frac{1}{\pi}\,{\rm Im} \; e^{i\pi\nu}\int_0^\infty e^{-tu} \,  u^{-\nu} \, e^{- \lambda (e^{-i\pi}u)^\alpha}  \, du 
 \label{eq:rhoconvstable3} \\
&=  \frac{1}{\pi} \sum_{k=0}^\infty \frac{(-\lambda )^k}{k!} \sin\pi(\nu-\alpha k) \frac{\Gamma(\alpha k-\nu+1)}{t^{\alpha k-\nu+1}}
 \label{eq:rhoconvstableinfsum}  \\
 &=  \sum_{k=0}^\infty \frac{(-\lambda)^k}{k!} \frac{t^{\nu-\alpha k-1}}{\Gamma(\nu-\alpha k)}   
 \equiv  \sum_{k=0}^\infty \frac{(-\lambda)^k}{k!} \rho_{\nu-\alpha k}(t)   
 \label{eq:rhoconvstableinfsum1} 
\end{align}
\end{theorem}
$\nu=0$  
reproduces the  representation~(\ref{eq:Pollardinfsum2}) of the stable density
$f_\alpha(t\vert\lambda) \equiv \{\rho_0\star f_\alpha(\cdot\vert \lambda\}(t)$.

\begin{remark}
$(\ref{eq:rhoconvstableinfsum1})$ may also be written as 
 $\{\rho_\nu\star f_\alpha(\cdot\vert \lambda) \}(t) = t^{\nu-1} \phi(-\alpha,\nu,-\lambda t^{-\alpha})$,
 where $\phi$ 
 is the Wright function that finds much application in fractional calculus.
 Its definition as an entire function  for $z\in \mathbb{C}$  is ({\it e.g.}\  Gorenflo {\it et al.}~\cite[F.2]{gorenflo2014mittag})
\begin{align}
\phi(\alpha,\beta,z)
 &=  \sum_{k=0}^\infty \frac{z^k}{k! \Gamma(\alpha  k+\beta)}   \qquad \alpha>-1, \beta \in \mathbb{C}
\label{eq:Wright} 
\end{align}
\end{remark}

\begin{proof}[Proof of Theorem~\ref{thm:rhoconvstable}]
Corollary~\ref{cor:CompositionLaplaceconv} gives~(\ref{eq:rhoconvstable1}).
Substituting~(\ref{eq:rhoconvdelta}) in~(\ref{eq:rhoconvstable1})  gives~(\ref{eq:rhoconvstable2}).
Given that the  Laplace transform of~(\ref{eq:rhoconvdelta})   is  $x^{-\nu}\psi(x) = x^{-\nu}e^{-x}$,
Corollary~\ref{cor:CompositionLaplacePollard} gives~(\ref{eq:rhoconvstable3}). Explicitly
 \begin{align*}
 \{\rho_\nu\star f_\alpha(\cdot\vert \lambda\}(t)  
 &=  \lambda^{\nu/\alpha}  \int_0^\infty f_\alpha(t\vert \lambda u) \rho_{\nu/\alpha}(u-1) \mathbbm{1}_{[1,\infty)} \, du  \\
 &= \frac{1}{\pi}\,{\rm Im}  \int_0^\infty e^{-tu} \, (e^{-i\pi\alpha}u^\alpha)^{-\nu/\alpha} \, e^{- \lambda (e^{-i\pi}u)^\alpha}  \, du  \\
 &=  \frac{1}{\pi}\,{\rm Im}  \sum_{k=0}^\infty \frac{(-\lambda )^k}{k!}  e^{i\pi(\nu-\alpha k)} \int_0^\infty e^{-tu} u^{\alpha k-\nu} \, du 
\intertext{
The   integral may be evaluated  for $\alpha k-\nu+1>0$, which only holds for all $k$ if $\nu<1$, to give}
 &=  \frac{1}{\pi}   \sum_{k=0}^\infty \frac{(-\lambda )^k}{k!}  \sin\pi(\nu-\alpha k) \frac{\Gamma(\alpha k-\nu+1)}{t^{\alpha k-\nu+1}} \\
 &=  \sum_{k=0}^\infty \frac{(-\lambda )^k}{k!} \frac{t^{\nu-\alpha k-1}}{\Gamma(\nu-\alpha k)}   
\end{align*}
The last step follows from the Euler reflection formula $\sin(\pi z)=\pi/\Gamma(z)\Gamma(1-z)$, where
$1/\Gamma(z)$ is an entire function for $z\in \mathbb{C}$ by   analytic continuation.
Hence  $1/\Gamma(\nu-\alpha k)$ is well-defined for $\nu\ge0$, without the   $\nu<1$ restriction  
imposed to evaluate the intermediate integral.
\end{proof}

This shows that the convolution  density  $\{\rho_\nu\star f_\alpha(\cdot\vert \lambda)\}(t)$ 
may  be replaced by   a mixture involving the simple convolution between $\rho_{\nu/\alpha}$ and a $\delta$ density.
In turn, the mixture  may optionally be represented as an infinite sum of (analytically continued) copies 
$\{\rho_{\nu-\alpha k}(t)\}$ of the power density $\rho_\nu(t)$.
$\{\rho_\nu\star f_\alpha(\cdot\vert \lambda)\}(t)$  is the basic building block for all other densities discussed in the balance of this paper.
In the case of $\nu=1-\alpha$, we have already  seen from Proposition~\ref{prop:stable} that the convolution has a simple form
 $\{\rho_{1-\alpha}\star f_\alpha(\cdot\vert \lambda)\}(t)=tf_\alpha(t\vert \lambda)/\lambda\alpha$.

In addition to the forms $(\ref{eq:rhoconvstable1})-(\ref{eq:rhoconvstableinfsum1})$, 
$\{\rho_\nu\star f_\alpha(\cdot\vert \lambda) \}(t)$ has an equivalent representation involving the beta distribution.
\begin{corollary}
\label{cor:rhoconvstable}
The convolution density $\{\rho_\nu\star f_\alpha(\cdot\vert \lambda) \}(t)$ can  be written as 
\begin{align}
 \{\rho_\nu\star f_\alpha(\cdot\vert \lambda) \}(t)  
 &=  \frac{\lambda^{\nu/\alpha}}{\Gamma(\nu/\alpha)}
  \int_0^1 f_\alpha(t\vert \lambda/u) \, u^{-\nu/\alpha}  (1-u)^{\nu/\alpha-1}  \, \frac{du}{u} 
  \label{eq:betastabledensity1} \\
 &\equiv \Gamma\left(1-\tfrac{\nu}{\alpha}\right)  \lambda^{\nu/\alpha}
  \int_0^1 f_\alpha(t\vert \lambda/u) \, {\rm beta}\left(u\vert 1-\tfrac{\nu}{\alpha}, \tfrac{\nu}{\alpha}\right)  \frac{du}{u}  
 \label{eq:betastabledensity2}
 \end{align}
\end{corollary}

This has the form of (\ref{eq:betaVdensity}), the density  of a product of two variables where one of them is beta distributed.
We shall pursue this fundamental   result further later.
It suffices to note for now that there are various ways to represent  the convolution $\{\rho_\nu\star f_\alpha(\cdot\vert \lambda) \}(t) $.

\begin{proof}[Proof of Corollary~\ref{cor:rhoconvstable}]
By (\ref{eq:rhoconvstable2})
\begin{align*}
\{\rho_\nu\star f_\alpha(\cdot\vert \lambda) \}(t) 
 &=  \lambda^{\nu/\alpha}  \int_1^\infty f_\alpha(t\vert \lambda u) \rho_{\nu/\alpha}(u-1) \, du  \\
 (u\to1/u)\quad
 &=  \lambda^{\nu/\alpha}  \int_0^1 f_\alpha(t\vert \lambda/u) \rho_{\nu/\alpha}(1/u-1) \, \frac{du}{u^2} \\
  &=  \frac{\lambda^{\nu/\alpha}}{\Gamma(\nu/\alpha)}
  \int_0^1 f_\alpha(t\vert \lambda/u) \, u^{-\nu/\alpha}  (1-u)^{\nu/\alpha-1}  \, \frac{du}{u} 
 \end{align*}
 thereby proving (\ref{eq:betastabledensity1}), hence (\ref{eq:betastabledensity2}).
\end{proof}
For  $\lambda=1$, (\ref{eq:betastabledensity1}) is equivalent  to Ho~{\it  et al.}~\cite[Proposition~2.2({\it iii})]{HoJamesLau},
derived by different reasoning (with the convolution $\rho_\nu\star f_\alpha$ expressed as  the 
Riemann-Liouville fractional integral  $I_+^\nu\, f_\alpha$).

We may now turn to other  choices of  $F$, starting  with the exponential distribution.
The results presented for this case are almost entirely well-known, the point here is the manner of their derivation using convolutions and mixtures 
and the insights arising therefrom.


 \section{Exponentially   Distributed   $F$}
 \label{sec:expF}
 
 Let $F\equiv G(1,y)$  with density $\rho_{1,y}$ be the exponential distribution,  a particular  case of   the gamma distribution defined in 
 Section~\ref{sec:gamma}, with shape parameter  $\gamma=1$ and scale parameter $y>0$. 
Explicitly, $dF(t\vert y) \equiv \rho_{1,y}(t)dt =y e^{-yt}dt$ with Laplace-Stieltjes transform 
\begin{align}
\psi( x \vert y)  &= \int_0^\infty  e^{-x t} \, dF(t\vert y)=  \frac{y}{y+ x} \quad
\implies \psi( x^\alpha\vert y)  
=  \frac{y}{y+ x^\alpha} 
\label{eq:expLTcomposition} 
\end{align}
For $0<\alpha<1$, 
Theorem~\ref{thm:CompositionLaplace} states that $\psi( x^\alpha\vert y)$ 
is the Laplace transform of  a mixture density} 
\begin{align}
\psi( x^\alpha\vert y)  &= \int_0^\infty e^{-x t} \, w_\alpha(t\vert y)\, dt \\
\textrm{where}\quad w_\alpha(t\vert y) 
&=  \int_0^\infty   f_\alpha(t\vert  u) \,dF(u\vert y)  
\label{eq:wexp}
 \end{align}
 
\begin{theorem}
\label{thm:MLfunction}
Let $0<\alpha<1$, $w_\alpha(t\vert y)$  given by~$(\ref{eq:wexp})$
and $\rho_{1-\alpha}(t)=t^{-\alpha}/\Gamma(1-\alpha)$.
Then the  Mittag-Leffler function  $E_\alpha(-y t^\alpha)$  has the  convolution representation
\begin{align}
y E_\alpha(-y t^\alpha) &=   \{\rho_{1-\alpha}\star w_\alpha(\cdot\vert y)\}(t) \quad t>0, y>0
\label{eq:MLconv} \\
&\equiv  \int_0^\infty   \{\rho_{1-\alpha}\star  f_\alpha(\cdot \vert  u)\}(t) \,dF(u\vert y)  
\label{eq:MLconv1} \\
&=  \int_0^\infty  f_\alpha(t\vert  u)  \{\rho_{1/\alpha-1} \star \rho_{1,y}\}(u) \, du 
\label{eq:MLconv2}
\end{align}
where $dF(u\vert y)\equiv \rho_{1,y}(u) du = y e^{-yu}du$.
\end{theorem}

\begin{proof}[Proof of Theorem  \ref{thm:MLfunction}]
The convolution $\{\rho_{1-\alpha}\star w_\alpha(\cdot \vert y)\}(t)$ of $\rho_{1-\alpha}(t)$  (Laplace  transform  $x^{\alpha-1}$) 
and  $w_\alpha(t\vert y)$ gives the following
\begin{align*}
x^{\alpha-1} \psi( x^\alpha\vert y) 
=  y \, \frac{x^{\alpha-1}}{y+ x^\alpha} 
&= y \int_0^\infty e^{-x t} \, E_\alpha(-y t^\alpha) \, dt  \\
&= \int_0^\infty e^{-x t} \, \{\rho_{1-\alpha}\star w_\alpha(\cdot \vert y)\}(t) \, dt  
\end{align*}
where  the Laplace transform of the Mittag-Leffler function $E_\alpha(-y t^\alpha)$ is read from~(\ref{eq:LaplaceML3par}).
Hence $y E_\alpha(-y t^\alpha) =   \{\rho_{1-\alpha}\star w_\alpha(\cdot\vert y)\}(t)$, which is identical to~(\ref{eq:MLconv1}) by 
the definition~(\ref{eq:wexp}) of  $w_\alpha(t\vert y)$.
The equality of (\ref{eq:MLconv1})  and (\ref{eq:MLconv2}) follows from 
Corollary~\ref{cor:CompositionLaplaceconv} of Theorem~\ref{thm:CompositionLaplace}.
\end{proof}

While  (\ref{eq:MLconv1})  and (\ref{eq:MLconv2}) are formally equivalent, 
they lead to  different outcomes.
We interpret $E_\alpha(-y t^\alpha)$ as a function of two variables $t$ and $y$.  
Corollary~\ref{cor:MLconv1} treats~ (\ref{eq:MLconv1})  as a function of $y$ given $t$.
This leads to the  Mittag-Leffler distribution, expressed as a convolution.
Corollary~\ref{cor:MLconv1Composition} gives a  consequence  thereof.
Corollary~\ref{cor:MLconv2}, on the other hand, treats~(\ref{eq:MLconv2}) as a function of $t$ given $y$,
leading directly to the result of Corollary~\ref{cor:MLconv1Composition}, but without the ability to reproduce the underlying results of 
Corollary~\ref{cor:MLconv1} concerning the Mittag-Leffler distribution.


\subsection{Mittag-Leffler distribution}
\label{sec:ML}

\begin{corollary}[Mittag-Leffler distribution as a convolution]
\label{cor:MLconv1}
For $0<\alpha<1$,   $E_\alpha(-yt^\alpha)$ as a function of $y$  given $t>0$  is the Laplace-Stieltjes transform of a   distribution 
$P_\alpha(u\vert t)$
\begin{align}
E_\alpha(-y t^\alpha) &=    \int_0^\infty    e^{-yu} \,  dP_\alpha(u\vert t) \quad (y\ge0)
\label{eq:MLconv1LSt} \\
{\rm where}\quad 
 dP_\alpha(u\vert t) &= \{\rho_{1-\alpha}\star  f_\alpha(\cdot \vert  u)\}(t) \, du  \\
 &=\frac{t}{\alpha}  f_\alpha(t \vert  u)  \,  u^{-1}  \, du 
  \equiv    \frac{t}{\alpha}  f_\alpha(t u^{-1/\alpha}) \,  u^{-1/\alpha-1} \, du 
 \label{eq:MLconv1distribution} 
\intertext{For $t=1$, $P_\alpha(u)\equiv P_\alpha(u\vert 1)$ with density $p_\alpha(u)\equiv p_\alpha(u\vert 1)$ is known as  the  
Mittag-Leffler distribution with Laplace-Stieltjes transform $E_\alpha(-y)$}
E_\alpha(-y) &=    \int_0^\infty    e^{-yu} \,  dP_\alpha(u)  = \int_0^\infty    e^{-yu} \,  p_\alpha(u) \, du
\label{eq:LStMLdistribution} \\
{\rm where}\quad 
 p_\alpha(u) &= \{\rho_{1-\alpha}\star  f_\alpha(\cdot \vert  u)\}(1) 
 \label{eq:MLdistribution} \\
 &=\frac{1}{\alpha}  f_\alpha(1 \vert  u)  \,  u^{-1}   \equiv    \frac{1}{\alpha}  f_\alpha( u^{-1/\alpha}) \,  u^{-1/\alpha-1} \, 
 \label{eq:MLdistribution1} \\
 &=  \frac{1}{\pi\alpha u}\,{\rm Im}  \int_0^\infty e^{-v} \, e^{-u (e^{-i\pi}v)^\alpha}  \, dv  
\label{eq:MLdistributionPollard} \\
 &= \frac{1}{\pi\alpha}  \sum_{k=1}^\infty \frac{(-u)^{k-1}}{k!} \sin(\pi\alpha k) \, \Gamma(\alpha k+1) 
\label{eq:MLdistributionPollardinfsum} 
\end{align}
The moments of $P_\alpha$ are
\begin{align}
M_{\alpha;k} &=  \int_0^\infty u^k \, dP_\alpha(u) \equiv (-1)^k E^{(k)}_\alpha(0) = \frac{k!}{\Gamma(\alpha k+1)} \quad k\ge0
\label{eq:MLmoments}
\end{align}
\end{corollary}

\begin{proof}[Proof of Corollary   \ref{cor:MLconv1}]
Given that  $dF(u\vert y) = y e^{-yu}du$, $(\ref{eq:MLconv1})$ gives
\begin{align*}
E_\alpha(-y t^\alpha) &=    \int_0^\infty    e^{-yu} \, \{\rho_{1-\alpha}\star  f_\alpha(\cdot \vert  u)\}(t) \, du  
\equiv  \int_0^\infty    e^{-yu} \,  dP_\alpha(u\vert t) \quad (t>0)
\intertext{where $dP_\alpha(u\vert t) = \{\rho_{1-\alpha}\star  f_\alpha(\cdot \vert  u)\}(t)  du$.
By Proposition~\ref{prop:stable},  $\alpha u \, \{\rho_{1-\alpha}\star f_\alpha(\cdot\vert u)\}(t) = tf_\alpha(t\vert u)$.
Hence}
 dP_\alpha(u\vert t) &= \{\rho_{1-\alpha}\star  f_\alpha(\cdot \vert  u)\}(t) \, du =\frac{t}{\alpha}  f_\alpha(t \vert  u)  \,  u^{-1}  \, du 
  \equiv    \frac{t}{\alpha}  f_\alpha(t u^{-1/\alpha}) \,  u^{-1/\alpha-1} \, du 
\end{align*}
The Mittag-Leffler distribution is $P_\alpha(u)\equiv P_\alpha(u\vert 1)$. 
The forms of the density given by~(\ref{eq:MLdistributionPollard}) and~(\ref{eq:MLdistributionPollardinfsum}) follow
from~(\ref{eq:Pollard}) and~(\ref{eq:Pollardinfsum2}) respectively. 
The moments of $P_\alpha(u)$   follow easily.
\end{proof}
As discussed in Section~\ref{sec:MLdistributions},
Pollard~\cite{PollardML}  used  complex analysis  to derive  $P_\alpha$. 
In a probabilistic  derivation,
Feller~\cite[XIII.8]{Feller2} used the two-dimensional Laplace  transform of the 
one-sided  $\alpha$-stable distribution with an associated  scale factor. 
Our approach here is inspired by Feller's approach, with convolution as a running theme throughout the paper.
Hence the density  takes a convolution  form equivalent to the familiar  form given by~(\ref{eq:MLdistribution1}).

\begin{corollary}[Composition]
\label{cor:MLconv1Composition}
For $0<\alpha<1$ and  $0<\sigma<1$, the composition $E_\alpha(-\lambda x^\sigma)$ as a function of $x$ given $\lambda>0$  
is the Laplace transform of a mixture  density
$w_{\sigma,\alpha}(t\vert \lambda)$
\begin{align}
E_\alpha(-\lambda x^\sigma) &= \int_0^\infty  e^{-x t} \,  w_{\sigma,\alpha}(t\vert\lambda) \,  dt \quad (x\ge0) 
\label{eq:LStMLComposition} \\
\textrm{where}\quad w_{\sigma,\alpha}(t\vert \lambda)&=  \int_0^\infty   f_\sigma(t\vert  \lambda u) \,dP_\alpha(u) 
\label{eq:wML} \\
&= \phantom{-} \frac{1}{\pi} {\rm Im} \int_0^\infty e^{-tu} \, E_\alpha(-\lambda e^{-i\pi\sigma}u^\sigma) \, du  
\label{eq:wMLPollard} \\
 &= -\frac{1}{\pi} \sum_{k=0}^\infty (-\lambda)^k\sin(\pi\sigma k) \, \frac{\Gamma(\sigma k+1)}{\Gamma(\alpha k+1)} \, t^{-\sigma k-1} 
\label{eq:wMLPollardseries} 
\intertext{For $\sigma=\alpha$, $E_\alpha(-\lambda x^\alpha)$ is the Laplace transform of  the mixture  density $w_{\alpha,\alpha}(t\vert\lambda)$}
w_{\alpha,\alpha}(t\vert\lambda)&=  \int_0^\infty   f_\alpha(t\vert  \lambda u) \,dP_\alpha(u)  
\label{eq:wMLalpha} \\
 &=  \frac{\sin\pi\alpha}{\pi}\, \frac{\lambda\, t^{\alpha-1}} {\lambda^2+2\lambda \, t^\alpha \cos\pi\alpha+t^{2\alpha}} 
\label{eq:wMLalphaPollard} 
 \end{align} 
\end{corollary}
\begin{proof}[Proof of Corollary  \ref{cor:MLconv1Composition}]
By Theorem~\ref{thm:CompositionLaplace} with $F(t)\equiv P_\alpha(t)$ and  $\psi(x)\equiv E_\alpha(-x)$,  
$\psi(\lambda x^\sigma) \equiv E_\alpha(-\lambda x^\sigma)$ is  the Laplace  transform of a mixture  density given by~$(\ref{eq:wML})$.
The equivalent forms~$(\ref{eq:wMLPollard})$ and $(\ref{eq:wMLPollardseries})$ follow from Corollaries~\ref{cor:CompositionLaplacePollard} and
\ref{cor:CompositionLaplaceInfseries} respectively.
For $\sigma=\alpha$,~(\ref{eq:wML}) becomes
\begin{alignat*}{3}
w_{\alpha,\alpha}(t\vert\lambda) &=  \int_0^\infty   f_\alpha(t\vert  \lambda u) \,dP_\alpha(u) \\
 &= \phantom{-} \frac{1}{\pi} {\rm Im} \int_0^\infty e^{-tu} \, E_\alpha(-\lambda e^{-i\pi\alpha}u^\alpha) \, du  
 &&=  \frac{1}{\pi} \, {\rm Im}\, \frac{t^{\alpha-1}}{\lambda e^{-i\pi\alpha}+t^\alpha} \\
\textrm{or}\;  &= -\frac{1}{\pi} \sum_{k=0}^\infty (-\lambda)^k\sin(\pi\alpha k) \,  t^{-\alpha k-1}
 &&=   \frac{1}{\pi t}\, {\rm Im}\, \frac{1}{1+ \lambda e^{-i\pi\alpha}t^{-\alpha}}  \\
 &=  \frac{\sin\pi\alpha}{\pi}\, \frac{\lambda\, t^{\alpha-1}} {\lambda^2+2\lambda \, t^\alpha \cos\pi\alpha+t^{2\alpha}} 
 \end{alignat*} 
with the Laplace  transform  and the  geometric  sum  leading to the same  result.
\end{proof}
The fundamental role of the Mittag-Leffler distribution $P_\alpha$ is thus evident.
$E_\alpha(-x)$ and the compositions   $E_\alpha(-\lambda x^\sigma),  E_\alpha(-\lambda x^\alpha)$ arise from  
$P_\alpha$  and mixtures of $f_\sigma, f_\alpha$  with respect to  $P_\alpha$.

\begin{remark}
\label{rem:rvratios}
The $\sigma=\alpha$   density $w_{\alpha,\alpha}(t\vert\lambda)$ of~(\ref{eq:wMLalphaPollard}) is  well-known in the probabilistic  literature.
It  has an intimate association with  occupation times in Brownian motion (Bertoin~{\it et al.}~\cite{Bertoin}, Lamperti~\cite{Lamperti})).
For $\lambda =1$,  (\ref{eq:wMLalphaPollard}) arises   as the density   of the ratio 
$X/Y$ of  independent identically distributed ({\it iid}) $\alpha$-stable random variables $X,Y$  
({\it e.g.}\  James~\cite{James_Lamperti}, James  {\it et al.}~\cite{JamesRoynetteYor}).
In Appendix~\ref{sec:MLrv}, we demonstrate that a slight variant on $X/Y$ (with $\lambda>0$) 
has a  density that takes the  mixture integral form of~(\ref{eq:wMLalpha}),
which reduces to the more commonly cited second form~(\ref{eq:wMLalphaPollard}) 
under the Pollard representation~(\ref{eq:Pollardinfsum2}) of the stable density $f_\alpha(t \vert  \lambda u)$.
 We proceed to show that the  density of   $T=(X/Y)^\alpha$ for   {\it iid} $\alpha$-stable random variables $X,Y$  is
\begin{align}
\frac{1}{\alpha t} w_{\alpha,\alpha}(1\vert t) 
&=  \frac{1}{\alpha t}  \int_0^\infty   f_\alpha(1 \vert  t u) \,dP_\alpha(u) 
 \label{eq:ratiopowerdensity} \\
 &=  \frac{\sin\pi\alpha}{\pi\alpha}\, \frac{1} {t^2+2t  \cos\pi\alpha+1}  \qquad (t\ge0)
 \label{eq:ChamYor}
 \end{align} 
The form (\ref{eq:ChamYor})  reproduces  Chaumont and Yor~\cite[p116(4.21.3)]{ChaumontYor} via the mixture density 
$w_{\alpha,\alpha}(1\vert t)$ and  the  representation~(\ref{eq:Pollardinfsum2}) of the stable density
$f_\alpha(1 \vert  t u)$.
\end{remark}

\begin{remark}
$w_{\alpha,\alpha}(t\vert\lambda)$  is also the density of the Thorin measure 
(in the language of infinitely divisible distributions known as generalised gamma convolutions)
of the distribution of the product $X Y^{1/\alpha}$ where $\{X,Y\}$ are respectively the $\alpha$-stable   and
gamma random variables (Bondesson~\cite[3.2.4 (p38)]{Bondesson}).
Although Bondesson did not use the term, the distribution in question is the generalised positive  Linnik distribution,
as named by  Pakes~\cite{Pakes95}.
For completenesses, we discuss next  the case where $Y$ is exponentially distributed,
deferring the more general gamma distribution  case to later discussion.
\end{remark}

\subsection{Positive  Linnik Distribution}
\label{sec:ML}

As noted in Section~\ref{sec:otherMLdistributions}, 
$\lambda^\gamma/(\lambda+x^\alpha)^\gamma$, ($0<\alpha<1, \gamma>0,\lambda>0)$ 
is the Laplace-Stieltjes  transform of  the generalised  positive  Linnik distribution ${\rm PL}(\alpha,\gamma,\lambda)$.  
As further noted,  $\gamma=1$ defines the  positive  Linnik distribution ${\rm PL}(\alpha,\lambda)\equiv {\rm PL}(\alpha,1,\lambda)$ with explicit form
${\rm PL}(t\vert \alpha,\lambda)=1-E_\alpha(-\lambda t^\alpha)$.
It is also well-known that the corresponding  density may be written  
as the mixture density $w_\alpha(t\vert\lambda)$  of~(\ref{eq:wexp}), 
whose  Laplace transform is  $\lambda/(\lambda+x^\alpha)$.
There is yet another representation of this density in terms of  the  mixture density $w_{\alpha,\alpha}(t\vert\lambda)$  of~(\ref{eq:wMLalpha})
in Corollary~\ref{cor:MLconv1Composition}:

\begin{corollary}[Positive Linnik distribution]
\label{cor:PositiveLinnik}
Let ${\rm PL}(\alpha,\lambda)\equiv {\rm PL}(\alpha,1,\lambda)$ 
$(0<\alpha<1$ and $\lambda>0)$ be the positive Linnik distribution  with 
Laplace-Stieltjes  transform $\lambda/(\lambda+x^\alpha)$.
Then the density of ${\rm PL}(\alpha,\lambda)$ is 
\begin{align}
w_\alpha(t\vert\lambda) 
 &= \lambda \int_0^\infty   f_\alpha(t\vert  u) \, e^{-\lambda u}\, du 
   = -\dfrac{d}{dt}E_{\alpha}(-\lambda t^\alpha)  
= \int_0^\infty e^{-tu} \, u \, w_{\alpha,\alpha}(u \vert \lambda) \, du
\label{eq:PLdensity} 
\end{align}
where $w_{\alpha,\alpha}(u\vert\lambda)$   of~$(\ref{eq:wMLalpha})$ is the mixture   of the $\alpha$-stable density with respect to the 
Mittag-Leffler distribution  $P_\alpha$ of Corollary~$\ref{cor:MLconv1}$.
\end{corollary}
\begin{proof}[Proof of Corollary~\ref{cor:PositiveLinnik}]
The first expression of $w_\alpha(t\vert\lambda)$is the well-known mixture of~$(\ref{eq:wexp})$.
The second  expression is equally well-known, being the derivative of the positive  Linnik  distribution 
${\rm PL}(t\vert \alpha,\lambda)=1-E_\alpha(-\lambda t^\alpha)$.
The third  expression is the Laplace transform of $u  w_{\alpha,\alpha}(u \vert \lambda)$ induced by the
fact that $E_\alpha(-\lambda t^\alpha)$ is the Laplace transform of $w_{\alpha,\alpha}(u \vert \lambda)$, 
as derived in Corollary~\ref{cor:MLconv1Composition}.
\end{proof}
Corollary~\ref{cor:PositiveLinnik} shows that the positive  Linnik  distribution  ${\rm PL}(\alpha,\lambda)$ 
may be looked upon as  induced by the  Mittag-Leffler distribution  $P_\alpha$ 
through the mixture  $w_{\alpha,\alpha}(\cdot \vert \lambda)$.
$P_\alpha$ itself may have been derived by any approach, be it the triangular P{\' o}lya urn  model (Janson~\cite{Janson, Janson2024})
 or the convolution construction 
of Corollary~\ref{cor:MLconv1}.

Finally, Theorem~\ref{thm:MLfunction} also enables an alternative and direct construction  of $E_\alpha(-\lambda x^\alpha)$ 
as the Laplace transform of the density~(\ref{eq:wMLalpha}), as shown in Corollary~\ref{cor:MLconv2} below.
However, the Mittag-Leffler distribution $P_\alpha$ does not feature  here and cannot be  inferred from this construction.

\begin{corollary}
\label{cor:MLconv2}
For $0<\alpha<1$, the representation  $(\ref{eq:MLconv2})$ directly implies that 
$E_\alpha(-yt^\alpha)$ as a function of $t$  given $y>0$ is the Laplace transform of a   density $q_\alpha(u\vert y)$
\begin{align}
y E_\alpha(-y t^\alpha) &=  \int_0^\infty  f_\alpha(t\vert  u)  \{\rho_{1/\alpha-1} \star \rho_{1,y}\}(u) \, du \nonumber \\
\implies 
E_\alpha(-y t^\alpha) &=    \int_0^\infty    e^{-tu} \,  q_\alpha(u\vert y)\, du   \quad (t\ge0)
\label{eq:MLconv2LSt} \\
{\rm where}\quad 
 q_\alpha(u\vert y)  
 &=  \frac{1}{y\pi}\,{\rm Im}\left \{\widetilde\rho_{1/\alpha-1}(e^{-i\pi\alpha}u^\alpha) \,  \widetilde\rho_{1,y}(e^{-i\pi\alpha}u^\alpha)\right\} 
 \label{eq:MLconv2distribution0} \\
 &=  \frac{\sin\pi\alpha}{\pi}\, \frac{y\, u^{\alpha-1}} {y^2+2y \, u^\alpha \cos\pi\alpha+u^{2\alpha}} \, du   
\label{eq:MLconv2density}
\end{align}
\end{corollary}

\begin{proof}[Proof of Corollary~\ref{cor:MLconv2}]
The  Laplace  transforms of $\rho_{1/\alpha-1}(t)$  and $\rho_{1,y}(t)$ are $\widetilde\rho_{1/\alpha-1}(x)=x^{1-1/\alpha}$ and 
$\widetilde\rho_{1,y}(x)=y/(y+x)$ respectively.
Hence, using the representation (\ref{eq:CompositionLaplacePollardConv}) of  Corollary~\ref{cor:CompositionLaplacePollard}
in $(\ref{eq:MLconv2})$ gives 
\begin{align*}
y E_\alpha(-y t^\alpha) 
&=  \int_0^\infty  f_\alpha(t\vert  u)  \{\rho_{1/\alpha-1} \star \rho_{1,y}\}(u) \, du \\
 &=  \frac{1}{\pi}\,{\rm Im} \int_0^\infty e^{-tu}\, \widetilde\rho_{1/\alpha-1}(e^{-i\pi\alpha}u^\alpha) \,  \widetilde\rho_{1,y}(e^{-i\pi\alpha}u^\alpha) \, du \\
&=  \frac{y}{\pi}\,{\rm Im} \int_0^\infty e^{-tu}\, \frac{(e^{-i\pi\alpha}u^\alpha)^{1-1/\alpha}}{y+  e^{-i\pi\alpha}u^\alpha} \, du \\
\implies E_\alpha(-y t^\alpha) 
&= \frac{1}{\pi}\,{\rm Im} \int_0^\infty e^{-tu}\, \frac{(e^{-i\pi}u)^{\alpha-1}}{y+  (e^{-i\pi}u)^\alpha} \, du \\
&= \frac{\sin\pi\alpha}{\pi} \int_0^\infty e^{-tu}\, \frac{y\, u^{\alpha-1}}{y^2+2y \, u^\alpha \cos\pi\alpha+u^{2\alpha}} \, du
\end{align*}
thereby proving that $E_\alpha(-y t^\alpha)$ is the Laplace transform   of  the density (\ref{eq:MLconv2density}).
\end{proof}

We summarise the results of  Theorem~\ref{thm:MLfunction} in Table~\ref{table:MLDensities}.
We may now readily generalise the foregoing construction, based on the exponentially distributed $F$, to that
based on  the  gamma distribution with  general shape parameter.

 \section{Gamma  Distributed  $F$}
 \label{sec:gammaF}
 
 Let $F\equiv G(\gamma,y)$  with density  $\rho_{\gamma,y}$ be the   gamma distribution as defined in Section~\ref{sec:gamma},
  with shape parameter  $\gamma>0$ and scale parameter $y>0$. 
Explicitly, $dG(t\vert\gamma,y)=y^\gamma t^{\gamma-1}/ \Gamma(\gamma) e^{-yt}dt$
with  Laplace-Stieltjes transform  
\begin{align}
\psi( x\vert \gamma,y)   &= \int_0^\infty  e^{-x t} \,dG(t\vert \gamma,y) = \frac{y^\gamma}{ (y+ x)^{\gamma}}
\implies \psi( x^\alpha\vert \gamma,y)   
=  \frac{y^\gamma}{ (y+ x^\alpha)^{\gamma}}
\label{eq:gammaLTcomposition} 
\end{align}
For $0<\alpha<1$, Theorem~\ref{thm:CompositionLaplace} states that~(\ref{eq:gammaLTcomposition})
is the Laplace transform of  a mixture density 
\begin{align}
\psi( x^\alpha\vert \gamma,y)  &= \int_0^\infty e^{-x t} \, w_\alpha(t\vert \gamma,y)\, dt \\
\textrm{where}\quad w_\alpha(t\vert \gamma,y) &=  \int_0^\infty   f_\alpha(t\vert  u) \,dG(u\vert \gamma,y) 
\label{eq:wgamma}
\end{align}

\begin{theorem}
\label{thm:ML3parfunction}
Let $0<\alpha<1$, $w_\alpha(t\vert \gamma, y)$  $(\gamma>0)$ given by~$(\ref{eq:wgamma})$
and $\rho_{\beta-\alpha\gamma}(t)=t^{\beta-\alpha\gamma-1}/\Gamma(\beta-\alpha\gamma)$ for $\beta-\alpha\gamma>0$.
More compactly,  $0<\alpha<1$, $0<\alpha\gamma<\beta$.
Then $t^{\beta-1}E^\gamma_{\alpha,\beta}(-y t^\alpha)$, 
where  $E^\gamma_{\alpha,\beta}(\cdot)$  is the 3-parameter Mittag-Leffler or Prabhakar function,   
has the  convolution representation
\begin{align}
y^\gamma \, t^{\beta-1}E^\gamma_{\alpha,\beta}(-y t^\alpha) &=  \{ \rho_{\beta-\alpha\gamma}\star w_\alpha(\cdot\vert \gamma, y)\}(t) 
\label{eq:ML3parconv} \\
&\equiv  \int_0^\infty   \{\rho_{\beta-\alpha\gamma} \star  f_\alpha(\cdot \vert  u)\}(t) \,dG(u\vert \gamma,y)  
\label{eq:ML3parconv1} \\
&=  \int_0^\infty  f_\alpha(t\vert  u) \{\rho_{\beta/\alpha-\gamma} \star \rho_{\gamma,y}\}(u) \, du 
\label{eq:ML3parconv2}
\end{align}
\end{theorem}
\begin{proof}[Proof of Theorem  \ref{thm:ML3parfunction}]
The convolution $\{\rho_{\beta-\alpha\gamma}\star w_\alpha(\cdot \vert \gamma, y)\}(t)$ of $w_\alpha(t\vert \gamma, y)$ and  
$\rho_{\beta-\alpha\gamma}(t)$  (Laplace  transform  $x^{\alpha\gamma-\beta}$) gives the following
\begin{align*}
x^{\alpha\gamma-\beta} \psi( x^\alpha\vert \gamma, y)  =  y^\gamma \,  \frac{x^{\alpha\gamma-\beta}} {(y+ x^\alpha)^{\gamma}}
&= y^\gamma  \int_0^\infty e^{-x t} \, t^{\beta-1}E^\gamma_{\alpha,\beta}(-y t^\alpha) \, dt  \\
&= \int_0^\infty e^{-x t} \, \{\rho_{\beta-\alpha\gamma}\star w_\alpha(\cdot \vert \gamma, y)\}(t) \, dt  
\end{align*}
where  the Laplace transform of  $t^{\beta-1}E^\gamma_{\alpha,\beta}(-y t^\alpha)$ is read from~(\ref{eq:LaplaceML3par}).
Hence $y^\gamma t^{\beta-1}E^\gamma_{\alpha,\beta}(-y t^\alpha) =   
\{\rho_{\beta-\alpha\gamma}\star w_\alpha(\cdot\vert y)\}(t)$, which is identical to~(\ref{eq:ML3parconv1}) by 
the definition~(\ref{eq:wgamma}) of  $w_\alpha(t\vert \gamma, y)$.
The equality of (\ref{eq:ML3parconv1})  and (\ref{eq:ML3parconv2}) follows from 
Corollary~\ref{cor:CompositionLaplaceconv} of Theorem~\ref{thm:CompositionLaplace}.
\end{proof}

As in Theorem~\ref{thm:MLfunction} above, 
the equivalent  representations  (\ref{eq:ML3parconv1})  and (\ref{eq:ML3parconv2}) have different outcomes.
We start with  (\ref{eq:ML3parconv1}) which gives a  definition of the  3-parameter Mittag-Leffler distribution  as a generalisation of the single
parameter case.

\begin{corollary}[3-parameter Mittag-Leffler distribution]
\label{cor:ML3parconv1}
For $0<\alpha<1$, $0<\alpha\gamma<\beta$,  
$\Gamma(\beta)t^{\beta-1}E^\gamma_{\alpha,\beta}(-xt^\alpha)$  as a function of $x$ given $t>0$
is the Laplace-Stieltjes transform of a   3-parameter  
distribution $P_{\alpha,\beta,\gamma}(u \vert t)$ with density $p_{\alpha,\beta,\gamma}(u\vert t)$ 
\begin{align}
\Gamma(\beta)  t^{\beta-1} E^\gamma_{\alpha,\beta}(-xt^\alpha) &= \int_0^\infty e^{-xu} \, dP_{\alpha,\beta,\gamma}(u\vert t)  \quad (x\ge0)
\label{eq:LStML3parconv1distribution} \\
{\rm where}\quad p_{\alpha,\beta,\gamma}(u\vert t) 
&= \Gamma(\beta) \, \rho_\gamma(u) \{\rho_{\beta-\alpha\gamma}\star f_\alpha(\cdot\vert u)\}(t)
 \label{eq:ML3parconv1distribution} \\
&= \Gamma(\beta) \, \rho_\gamma(u) u^{\beta/\alpha-\gamma}  \int_0^\infty  f_\alpha(t \vert uv) \rho_{\beta/\alpha-\gamma}(v-1)\mathbbm{1}_{[1,\infty)} dv 
 \label{eq:ML3parconv1distribution1}  \\
  &= \frac{\Gamma(\beta)  \rho_\gamma(u)}{\pi}\,{\rm Im}  \int_0^\infty e^{-tv} \, (e^{-i\pi}v)^{\alpha\gamma-\beta} \, e^{-u (e^{-i\pi}v)^\alpha}  \, dv 
 \label{eq:ML3parconv1distribution2} \\
 &= \frac{\Gamma(\beta)}{\Gamma(\gamma)} \, u^{\gamma-1}  
       \sum_{k=0}^\infty \frac{(-u)^k}{k!} \frac{t^{\beta-\alpha\gamma-\alpha k-1}}{\Gamma(\beta-\alpha\gamma-\alpha k)}   
 \label{eq:ML3parconv1distributioninfsum}
 \end{align}
For $t=1$, $P_{\alpha,\beta,\gamma}(u)\equiv P_{\alpha,\beta,\gamma}(u\vert 1)$ 
with density $p_{\alpha,\beta,\gamma}(u)\equiv p_{\alpha,\beta,\gamma}(u\vert 1)$ 
is the 3-parameter Mittag-Leffler distribution with Laplace-Stieltjes transform  $\Gamma(\beta)E^\gamma_{\alpha,\beta}(-x)$
\begin{align}
\Gamma(\beta) E^\gamma_{\alpha,\beta}(-x) &=  \int_0^\infty    e^{-xu} \,  dP_{\alpha,\beta,\gamma}(u) 
= \int_0^\infty    e^{-xu} \,  p_{\alpha,\beta,\gamma}(u) \, du
\label{eq:LStML3pardistribution} \\
{\rm where}\quad 
 p_{\alpha,\beta,\gamma}(u) &= \Gamma(\beta) \, \rho_\gamma(u) \{\rho_{\beta-\alpha\gamma}\star f_\alpha(\cdot\vert u)\}(1)
 \label{eq:ML3pardistribution} \\
 &= \Gamma(\beta) \rho_\gamma(u) u^{\beta/\alpha-\gamma}   \int_0^\infty  f_\alpha(1 \vert uv) \rho_{\beta/\alpha-\gamma}(v-1)\mathbbm{1}_{[1,\infty)} dv
 \label{eq:ML3pardistribution1} \\
 &= \frac{\Gamma(\beta)  \rho_\gamma(u)}{\pi}\,{\rm Im}  \int_0^\infty e^{-v} \, (e^{-i\pi}v)^{\alpha\gamma-\beta} \, e^{-u (e^{-i\pi}v)^\alpha}  \, dv 
 \label{eq:ML3pardistribution2} \\
 &= \frac{\Gamma(\beta)}{\Gamma(\gamma)} \, u^{\gamma-1}  
      \sum_{k=0}^\infty \frac{(-u)^k}{k!} \frac{1}{\Gamma(\beta-\alpha\gamma-\alpha k)}   
 \label{eq:ML3pardistributioninfsum}
\end{align}
The moments $M_{\alpha,\beta,\gamma;k}$ $k\ge0$ of $P_{\alpha,\beta,\gamma}$ are
\begin{align}
 M_{\alpha,\beta,\gamma;k} &\equiv  \int_0^\infty u^k \, dP_{\alpha,\beta,\gamma}(u) 
= \frac{\Gamma(\beta)\Gamma(\gamma+k)}{\Gamma(\gamma)\Gamma(\alpha k+\beta)}  
\label{eq:ML3parmoments}
\end{align}
\end{corollary}
\begin{proof}[Proof of Corollary   \ref{cor:ML3parconv1}]
Using $dG(u\vert\gamma,y) 
= y^\gamma \rho_\gamma(u) e^{-yu}\,du$
 in~(\ref{eq:ML3parconv1})  and scaling by $\Gamma(\beta)$ gives~(\ref{eq:LStML3parconv1distribution})
 with the density $p_{\alpha,\beta,\gamma}(u\vert t)$ given by~(\ref{eq:ML3parconv1distribution}).
The equivalent forms (\ref{eq:ML3parconv1distribution1}), \ref{eq:ML3parconv1distribution2})  
and (\ref{eq:ML3parconv1distributioninfsum}) follow from expressing the convolution  density 
$\{\rho_{\beta-\alpha\gamma}\star f_\alpha(\cdot\vert u)\}(t)$ in the respective forms (\ref{eq:rhoconvstable2}),  (\ref{eq:rhoconvstable3}) 
and (\ref{eq:rhoconvstableinfsum}) of  Theorem~\ref{thm:rhoconvstable}.
Setting  $t=1$ gives~(\ref{eq:LStML3pardistribution}) with the density $p_{\alpha,\beta,\gamma}(u)$ given by~(\ref{eq:ML3pardistribution})
and corresponding equivalent forms.
The moments~(\ref{eq:ML3parmoments}) follow  as $ \Gamma(\beta) (-1)^k E^\gamma_{\alpha,\beta}{}^{(k)}(0)$.
\end{proof}

M\"{o}hle~\cite[Lemma~2]{Mohle} proved (\ref{eq:ML3pardistributioninfsum}) by reference  to  a representation based on the Wright function by 
Tomovski {\it et al.}~\cite[Theorem~2]{Tomovski}. 
M\"{o}hle also cites Flajolet~{\it  et al.}~\cite[(80)]{Flajolet} for the same result, obtained by complex analytic inversion.

\begin{corollary}[Composition]
\label{cor:ML3parconv1Composition}
For $0<\alpha<1, 0<\sigma<1$ and $0<\alpha\gamma<\beta$, the composition $\Gamma(\beta)E^\gamma_{\alpha,\beta}(-\lambda x^\sigma)$ 
as a function of $x$ given  $\lambda>0$ is the Laplace transform of a mixture  density $w_{\sigma,\alpha}(t\vert \lambda, \gamma, \beta)$
\begin{align}
\Gamma(\beta)E^\gamma_{\alpha,\beta}(-\lambda x^\sigma) &= \int_0^\infty  e^{-x t} \,  w_{\sigma,\alpha}(t\vert \lambda, \gamma, \beta)\,  dt \quad (x\ge0) 
\label{eq:LStML3parComposition} \\
\textrm{where}\quad w_{\sigma,\alpha}(t\vert \lambda, \gamma, \beta)&=  \int_0^\infty   f_\sigma(t\vert  \lambda u) \,dP_{\alpha,\beta,\gamma}(u) 
\label{eq:wML3par} \\
&= \phantom{-} \frac{\Gamma(\beta)}{\pi} \, {\rm Im} \int_0^\infty e^{-tu} \, E^\gamma_{\alpha,\beta}(-\lambda e^{-i\pi\sigma}u^\sigma) \, du  
\label{eq:wML3parPollard} \\
 &= -\frac{1}{\pi} \sum_{k=0}^\infty (-\lambda)^k\sin(\pi\sigma k) \, M_{\alpha,\beta,\gamma;k} \frac{\Gamma(\sigma k+1)}{t^{\sigma k+1}} 
\label{eq:wML3parPollardseries} 
 \end{align} 
\end{corollary}
\begin{proof}[Proof of Corollary  \ref{cor:ML3parconv1Composition}]
By Theorem~\ref{thm:CompositionLaplace} with $F(t)\equiv P_{\alpha,\beta,\gamma}(t)$ and  $\psi(x)\equiv\Gamma(\beta) E^\gamma_{\alpha,\beta}(-x)$,  
the composition $\psi(\lambda x^\sigma) \equiv \Gamma(\beta) E^\gamma_{\alpha,\beta}(-\lambda x^\sigma)$ is  the Laplace  transform of a mixture  density
 given by~$(\ref{eq:wML3par})$.
The equivalent forms~$(\ref{eq:wML3parPollard})$ and $(\ref{eq:wML3parPollardseries})$ follow from Corollaries~\ref{cor:CompositionLaplacePollard} and
\ref{cor:CompositionLaplaceInfseries} respectively.
\end{proof}

\begin{corollary}
\label{cor:ML3parconv2}
For $0<\alpha<1$ and $0<\alpha\gamma<\beta$,  the representation  $(\ref{eq:ML3parconv2})$  implies that 
$t^{\beta-1}E^\gamma_{\alpha,\beta}(-y t^\alpha)$ as a function of $t$  given $y>0$ is the Laplace transform of a   density 
$q_{\alpha,\beta,\gamma}(u\vert y)$ 
\begin{align}
y^\gamma \, t^{\beta-1}E^\gamma_{\alpha,\beta}(-y t^\alpha) 
&=  \int_0^\infty  f_\alpha(t\vert  u) \{\rho_{\beta/\alpha-\gamma} \star \rho_{\gamma,y}\}(u) \, du  \nonumber \\
\implies 
t^{\beta-1}E^\gamma_{\alpha,\beta}(-y t^\alpha) &=    \int_0^\infty    e^{-tu} \,  q_{\alpha,\beta,\gamma}(u\vert y) \, du   \quad (t\ge0)
\label{eq:ML3parconv2LSt} \\
{\rm where}\quad 
 q_{\alpha,\beta,\gamma}(u\vert y) 
 &=  \frac{1}{y^\gamma\pi}\,{\rm Im}\left \{\widetilde\rho_{\beta/\alpha-\gamma}(e^{-i\pi\alpha}u^\alpha) \,  \widetilde\rho_{\gamma,y}(e^{-i\pi\alpha}u^\alpha)\right\} 
 \label{eq:ML3parconv2distribution0} \\
&= \frac{1}{\pi}\,{\rm Im} \frac{(e^{-i\pi}u)^{\alpha\gamma-\beta}}{(y+  e^{-i\pi\alpha}u^\alpha)^\gamma}   
\label{eq:ML3parconv2distribution}
\end{align}
\end{corollary}

\begin{proof}[Proof of Corollary   \ref{cor:ML3parconv2}]
The  Laplace  transforms of $\rho_{\beta/\alpha-\gamma}(t)$  and $\rho_{\gamma,y}(t)$ are respectively 
$x^{\gamma-\beta/\alpha}$ and $y^\gamma/(y+x)^\gamma$. 
Hence, using the representation (\ref{eq:CompositionLaplacePollardConv}) of  Corollary~\ref{cor:CompositionLaplacePollard}
in $(\ref{eq:ML3parconv2})$ gives 
\begin{align*}
y^\gamma \, t^{\beta-1}E^\gamma_{\alpha,\beta}(-y t^\alpha)
&=  \int_0^\infty  f_\alpha(t\vert  u)  \{\rho_{\beta/\alpha-\gamma} \star \rho_{\gamma,y}\}(u) \, du \\
 &=  \frac{1}{\pi}\,{\rm Im} \int_0^\infty e^{-tu}\, \widetilde\rho_{\beta/\alpha-\gamma}(e^{-i\pi\alpha}u^\alpha) \, 
            \widetilde\rho_{\gamma,y}(e^{-i\pi\alpha}u^\alpha) \, du \\
&= \frac{y^\gamma}{\pi}\,{\rm Im} \int_0^\infty e^{-tu}\, 
      \frac{(e^{-i\pi\alpha}u^\alpha)^{\gamma-\beta/\alpha}}{(y+  e^{-i\pi\alpha}u^\alpha)^\gamma}\, du \\
\implies t^{\beta-1}E^\gamma_{\alpha,\beta}(-y t^\alpha)
&= \frac{1}{\pi}\,{\rm Im} \int_0^\infty e^{-tu}\, \frac{(e^{-i\pi}u)^{\alpha\gamma-\beta}}{(y+  e^{-i\pi\alpha}u^\alpha)^\gamma} \, du 
\end{align*}
thereby completing  the proof.
\end{proof}


At face value, the same result can be obtained from  the product of $t^{\beta-1}$ and the composition $E^\gamma_{\alpha,\beta}(-\lambda t^\sigma)$
of Corollary~\ref{cor:ML3parconv1Composition} for $\sigma=\alpha$.
However, $t^{\beta-1}$ then has to be the Laplace transform of the  density $\rho_{1-\beta}(u)=u^{-\beta}/\Gamma(1-\beta)$ for $\beta<1$, 
so that $t^{\beta-1}E^\gamma_{\alpha,\beta}(-\lambda t^\alpha)$ is the Laplace transform of the convolution of two  densities.
By contrast, the additional  condition $\beta<1$ is not needed  for $q_{\alpha,\beta,\gamma}(u\vert y)$ to be a density  in  
Corollary~\ref{cor:ML3parconv2} since it does not rely on looking upon 
$t^{\beta-1}E^\gamma_{\alpha,\beta}(-y t^\alpha)$ as a product of 2 separate completely monotone functions
$t^{\beta-1}$ and $E^\gamma_{\alpha,\beta}(-y t^\alpha)$.

As noted in Section~\ref{sec:otherComplex},
Capelas de Oliveira {\it et al.}~\cite{Oliveira} and Mainardi and Garrappa~\cite{MainardiGarrappa}  
used a complex analytic approach to prove that $t^{\beta-1}E^\gamma_{\alpha,\beta}(-y t^\alpha)$  is the Laplace  transform of a nonnegative function  
for $0<\alpha\le1$ and  $0<\alpha\gamma\le\beta\le1$. 
However, as argued above, the additional  condition $\beta\le1$ can be avoided.

It is easy to  verify that all results of the  single parameter case are special cases of the 3-parameter case.
Corollary~{\ref{cor:MLconv2} is the special case $\beta=\gamma=1$.

We summarise the results of  Theorem~\ref{thm:ML3parfunction} in Table~\ref{table:ML3parDensities}.
It is straightforward to generalise the 3-parameter case to a 4-parameter case, as discussed next.

\subsection{A 4-parameter Mittag-Leffler distribution}
\label{sec:ML4par}

We  note that $\beta-\alpha\gamma$ is invariant under $\beta\to\beta+\theta$ and $\gamma\to\gamma+\theta/\alpha$
for an additional parameter $\theta$.
$\gamma+\theta/\alpha$ is the modified shape parameter of the Gamma distribution   considered in Theorem~\ref{thm:ML3parfunction}
 and it must satisfy $\gamma+\theta/\alpha>0$ or $\theta>-\alpha\gamma$.
Hence we may state an extended form of Theorem~\ref{thm:ML3parfunction}, omitting repetition of proofs

\begin{theorem}
\label{thm:ML4parfunction}
Consider Theorem~\ref{thm:ML3parfunction} with  $\beta\to\beta+\theta$ and $\gamma\to\gamma+\theta/\alpha$ so that 
$\beta-\alpha\gamma$ remains invariant.
Then~$(\ref{eq:wgamma})$ becomes  $w_\alpha(t\vert \gamma+\theta/\alpha, y)$  with $\gamma+\theta/\alpha>0$.
More compactly,  $0<\alpha<1$, $-\theta<\alpha\gamma<\beta$.
Then $t^{\beta+\theta-1}E^{\gamma+\theta/\alpha}_{\alpha,\beta+\theta}(-y t^\alpha)$, 
has the  convolution representation
\begin{align}
y^{\gamma+\theta/\alpha} \, t^{\beta+\theta-1}E^{\gamma+\theta/\alpha}_{\alpha,\beta+\theta}(-y t^\alpha)
&=  \{ \rho_{\beta-\alpha\gamma}\star w_\alpha(\cdot\vert \gamma+\theta/\alpha, y)\}(t) 
\label{eq:ML4parconv} \\
&\equiv  \int_0^\infty   \{\rho_{\beta-\alpha\gamma} \star  f_\alpha(\cdot \vert  u)\}(t) \,dG(u\vert \gamma+\theta/\alpha,y)  
\label{eq:ML4parconv1} \\
&=  \int_0^\infty  f_\alpha(t\vert  u) \{\rho_{\beta/\alpha-\gamma} \star \rho_{\gamma+\theta/\alpha,y}\}(u) \, du 
\label{eq:ML4parconv2}
\end{align}
\end{theorem}

\begin{corollary}[4-parameter Mittag-Leffler distribution]
\label{cor:ML4parconv1}
For $0<\alpha<1$ and $-\theta<\alpha\gamma<\beta$,  
$\Gamma(\beta+\theta) t^{\beta+\theta-1} E^{\gamma+\theta/\alpha}_{\alpha,\beta+\theta}(-xt^\alpha)$ 
as a function of $x$ given $t>0$
is the Laplace-Stieltjes transform of a   4-parameter  
distribution $P_{\alpha,\beta,\gamma,\theta}(\cdot\vert t)$
with density $p_{\alpha,\beta,\gamma,\theta}(u\vert t)$
\begin{align}
\Gamma(\beta+\theta) t^{\beta+\theta-1} &E^{\gamma+\theta/\alpha}_{\alpha,\beta+\theta}(-xt^\alpha) 
= \int_0^\infty e^{-xu } \, dP_{\alpha,\beta,\gamma,\theta}(u\vert t)  \quad (x\ge0)
\label{eq:LStML4parconv1distribution} \\
{\rm where}\; p_{\alpha,\beta,\gamma,\theta}(u\vert t) 
&= \Gamma(\beta+\theta) \, \rho_{\gamma+\theta/\alpha}(u) \{\rho_{\beta-\alpha\gamma}\star f_\alpha(\cdot\vert u)\}(t) 
 \label{eq:ML4parconv1distribution} \\
&= \Gamma(\beta+\theta) \, \rho_{\gamma+\theta/\alpha}(u) u^{\beta/\alpha-\gamma} 
    \int_0^\infty  f_\alpha(t \vert uv) \rho_{\beta/\alpha-\gamma}(v-1)\mathbbm{1}_{[1,\infty)} dv 
 \label{eq:ML4parconv1distribution1}  \\
  &= \frac{\Gamma(\beta+\theta)  \rho_{\gamma+\theta/\alpha}(u)}{\pi}\,{\rm Im}  \int_0^\infty e^{-tv} \, 
 (e^{-i\pi}v)^{\alpha\gamma-\beta} \, e^{-u (e^{-i\pi}v)^\alpha} \, dv 
 \label{eq:ML4parconv1distribution2} \\
 &= \frac{\Gamma(\beta+\theta)}{\Gamma(\gamma+\theta/\alpha)} \, u^{{\gamma+\theta/\alpha}-1}  
       \sum_{k=0}^\infty \frac{(-u)^k}{k!} \frac{t^{\beta-\alpha\gamma-\alpha k-1}}{\Gamma(\beta-\alpha\gamma-\alpha k)}   
 \label{eq:ML4parconv1distributioninfsum}
\end{align}
For $t=1$, $P_{\alpha,\beta,\gamma,\theta}(u)\equiv P_{\alpha,\beta,\gamma,\theta}(u\vert 1)$ 
with density $p_{\alpha,\beta,\gamma,\theta}(u)\equiv p_{\alpha,\beta,\gamma,\theta}(u\vert 1)$
is the 4-parameter Mittag-Leffler distribution 
with Laplace-Stieltjes transform  $\Gamma(\beta+\theta)E^{\gamma+\theta/\alpha}_{\alpha,\beta+\theta}(-x)$
\begin{align}
\Gamma(\beta+\theta) E^{\gamma+\theta/\alpha}_{\alpha,\beta+\theta}(-x) &=    \int_0^\infty    e^{-xu} \,  dP_{\alpha,\beta,\gamma,\theta}(u)
   =  \int_0^\infty    e^{-xu} \,  p_{\alpha,\beta,\gamma,\theta}(u)\, du
\label{eq:LStML4pardistribution} \\
{\rm where}\quad 
 p_{\alpha,\beta,\gamma,\theta}(u) &= \Gamma(\beta+\theta) \, \rho_{\gamma+\theta/\alpha}(u) \{\rho_{\beta-\alpha\gamma}\star f_\alpha(\cdot\vert u)\}(1) 
 \label{eq:ML4pardistribution} \\
 &= \Gamma(\beta+\theta) \rho_{\gamma+\theta/\alpha}(u) u^{\beta/\alpha-\gamma}  
  \int_0^\infty  f_\alpha(1 \vert uv) \rho_{\beta/\alpha-\gamma}(v-1)\mathbbm{1}_{[1,\infty)} dv
 \label{eq:ML4pardistribution1} \\
 &= \frac{\Gamma(\beta+\theta)  \rho_{\gamma+\theta/\alpha}(u)}{\pi}\,{\rm Im}  \int_0^\infty e^{-v} \, 
 (e^{-i\pi}v)^{\alpha\gamma-\beta} \, e^{-u (e^{-i\pi}v)^\alpha} \, dv 
 \label{eq:ML4pardistribution2} \\
 &= \frac{\Gamma(\beta+\theta)}{\Gamma(\gamma+\theta/\alpha)} \, u^{\gamma+\theta/\alpha-1}  
      \sum_{k=0}^\infty \frac{(-u)^k}{k!} \frac{1}{\Gamma(\beta-\alpha\gamma-\alpha k)}   
 \label{eq:ML4pardistributioninfsum}
\end{align}
The moments $M_{\alpha,\beta,\gamma,\theta;k}$ $k\ge0$ of $P_{\alpha,\beta,\gamma,\theta}$ are
\begin{align}
 M_{\alpha,\beta,\gamma,\theta;k} =  \int_0^\infty u^k \, dP_{\alpha,\beta,\gamma,\theta}(u) 
&= \frac{\Gamma(\beta+\theta)\Gamma(\gamma+\theta/\alpha+k)}{\Gamma(\gamma+\theta/\alpha)\Gamma(\alpha k+\beta+\theta)} 
\label{eq:ML4parmoments}
\end{align}
\end{corollary}

We note that, for general exponent $r>-\gamma-\theta/\alpha$, it can readily be shown that
\begin{align}
\int_0^\infty u^r \, dP_{\alpha,\beta,\gamma,\theta}(u) 
&= \frac{\Gamma(\beta+\theta)\Gamma(\gamma+\theta/\alpha+r)}{\Gamma(\gamma+\theta/\alpha)\Gamma(\alpha r+\beta+\theta)} 
\label{eq:ML4parmomentsgen}
\end{align}

The choice $\gamma=0\implies \beta>0, \theta>0$ gives  the  distribution  that  
Banderier~{\it  et al.}~\cite{Banderier} termed the beta-Mittag-Leffler distribution ${\rm BML(\alpha,\theta,\beta})$.
In this case,  the density~(\ref{eq:ML4pardistribution}) reduces to
\begin{align}
 p_{\alpha,\beta,\theta}(u) 
 &= \Gamma(\beta+\theta) \, \rho_{\theta/\alpha}(u) \{\rho_{\beta}\star f_\alpha(\cdot\vert u)\}(1)
 \label{eqBMLdensity1} \\
 &= \Gamma(\beta+\theta) \rho_{\theta/\alpha}(u) u^{\beta/\alpha}  
  \int_0^\infty  f_\alpha(1 \vert uv) \rho_{\beta/\alpha}(v-1)\mathbbm{1}_{[1,\infty)} dv 
 \label{eqBMLdensity2} \\
 &= \frac{\Gamma(\beta+\theta)}{\Gamma(\theta/\alpha)} \, u^{\theta/\alpha-1}  
      \sum_{k=0}^\infty \frac{(-u)^k}{k!} \frac{1}{\Gamma(\beta-\alpha k)}   
 \label{eqBMLdensityinfsum} 
\end{align}
The latter sum  reproduces~\cite[(25)]{Banderier}, obtained in that paper by the inverse Mellin transform.
The conceptual argument is   similar to the proof in Flajolet~{\it  et al.}~\cite[Proposition~16]{Flajolet}.

\begin{remark}
For $\beta=\gamma=1, \theta=0$, $(\ref{eq:ML4parmomentsgen})$ reduces to 
\begin{align}
\int_0^\infty u^r \, dP_\alpha(u) 
&= \frac{\Gamma(1+r)}{\Gamma(1+\alpha r)}    \qquad r >-1 
\label{eq:ML1parmomentsgen1}
\intertext{equivalently} 
\int_0^\infty u^{-r}\, dP_\alpha(u) 
&= \frac{\Gamma(1-r)}{\Gamma(1-\alpha r)}    \qquad r <1
\label{eq:ML1parmomentsgen2}
\end{align}
\end{remark}
We also note that
\begin{align}
\int_0^\infty u^r \, dP_\alpha(u) &\equiv \frac{1}{\alpha}  \int_0^\infty u^{r} f_\alpha(u^{-1/\alpha})u^{-1/\alpha-1}  \, du 
= \int_0^\infty u^{-\alpha r} \, f_\alpha(u) \, du
\label{eq:stablemomentsgen}
\end{align}
The equivalence of (\ref{eq:ML1parmomentsgen1}), (\ref{eq:ML1parmomentsgen2}),  (\ref{eq:stablemomentsgen})
 reproduces Shanbhag and Sreehari~\cite[Theorem~1]{Shanbhag}
 (also see Chaumont and Yor~\cite[p111(4.17.4)]{ChaumontYor})
from the perspective of Mittag-Leffler distributions.

\begin{corollary}[2-parameter Mittag-Leffler distribution]
\label{cor:ML2pardistribution}
For  $\theta>-\alpha\gamma$ and $\beta-\alpha\gamma=1-\alpha$, 
\begin{align}
dP_{\alpha,\beta,\gamma,\theta}(t)  &=   \Gamma(\beta+\theta) \,  \rho_{\gamma+\theta/\alpha}(t) \, dP_\alpha(t)
= \frac{\Gamma(\beta+\theta)}{\Gamma(\gamma+\theta/\alpha)} \, t^{\gamma+\theta/\alpha-1} \, dP_\alpha(t)
\label{eq:ML4parP1}  
\end{align} 
where $P_\alpha\equiv P_{\alpha,1,1,0}$ is the  Mittag-Leffler distribution. 
In particular, for $\beta=\gamma=1$, 
 $P_{\alpha,\theta}\equiv P_{\alpha,1,1,\theta}$ is the 2-parameter or generalized Mittag-Leffler distribution  ${\rm ML}(\alpha,\theta)$
 \begin{align}
d\,{\rm ML}(t\vert \alpha,\theta) \equiv dP_{\alpha,\theta}(t) 
&=   \Gamma(1+\theta) \,  
\rho_{1+\theta/\alpha}(t) \, dP_\alpha(t)  \nonumber \\
&= \frac{\Gamma(1+\theta)}{\Gamma(1+\theta/\alpha)} \, t^{\theta/\alpha} \, dP_\alpha(t)
\label{eq:ML2pardistribution}  
\end{align}
\end{corollary}
\begin{proof}[Proof of Corollary   \ref{cor:ML2pardistribution}]
\label{proof:ML2pardistribution}
For $\beta-\alpha\gamma=1-\alpha$, $\{\rho_{1-\alpha}\star f_\alpha(\cdot\vert  t)\}(1) \equiv (t\alpha)^{-1} f_\alpha(1\vert t)$ 
by Proposition~\ref{prop:stable},  so that~$(\ref{eq:ML4pardistribution})$ simplifies to 
\begin{align*}
dP_{\alpha,\beta,\gamma,\theta}(t) 
&= \Gamma(\beta+\theta) \,  \rho_{\gamma+\theta/\alpha}(t)  \, (t\alpha)^{-1}  f_\alpha(1\vert  t) \, dt  \\
&= \frac{\Gamma(\beta+\theta)}{\Gamma(\gamma+\theta/\alpha)} \, t^{\gamma+\theta/\alpha-1}
     \left[\frac{1}{\alpha} f_\alpha(t^{-1/\alpha}) \,  t^{-1/\alpha-1}\right] dt \\
&= \frac{\Gamma(\beta+\theta)}{\Gamma(\gamma+\theta/\alpha)} \, t^{\gamma+\theta/\alpha-1}\, dP_\alpha(t)
\end{align*}
where we have recalled that $f_\alpha(1\vert t) \equiv f_\alpha(t^{-1/\alpha}) \, t^{-1/\alpha}$.
$\beta=\gamma=1$ gives ${\rm ML}(\alpha,\theta)\equiv P_{\alpha,\theta}$.
\end{proof}

 The well-studied  generalised Mittag-Leffler distribution  ${\rm ML}(\alpha,\theta)$ was  discussed in Section~\ref{sec:MLdistributions}. 
 Using the same notation, we also  refer to $P_{\alpha,\beta,\gamma,\theta}$ as  ${\rm ML}(\alpha,\beta,\gamma,\theta)$.
 The general  condition $\beta>\alpha\gamma$  means that 
 the convolution $\{\rho_{\beta-\alpha\gamma}\star  f_\alpha(\cdot\vert  u)\}(t)$ is not reducible to a form proportional  to $f_\alpha(t\vert  u)$
 as it is in the  case $\beta-\alpha\gamma=1-\alpha$ for ${\rm ML}(\alpha,\theta)$, including $P_\alpha\equiv {\rm ML}(\alpha,0)$.
 
\begin{corollary}[Composition]
\label{cor:ML4parconv1Composition}
For $0<\alpha<1$, $0<\sigma<1$ and $-\theta<\alpha\gamma<\beta$,  
the composition $E^{\gamma+\theta/\alpha}_{\alpha,\beta+\theta}(-\lambda x^\sigma)$ as a function of $x$ given  $\lambda>0$
is the Laplace transform of a mixture  density $w_{\sigma,\alpha}(t\vert \lambda, \gamma, \beta,\theta)$
\begin{align}
\Gamma(\beta+\theta)E^{\gamma+\theta/\alpha}_{\alpha,\beta+\theta}(-\lambda x^\sigma) 
&= \int_0^\infty  e^{-x t} \,  w_{\sigma,\alpha}(t\vert \lambda, \gamma, \beta,\theta) \,  dt \quad (x\ge0) 
\label{eq:LStML4parComposition} \\
\textrm{where}\quad w_{\sigma,\alpha}(t\vert \lambda, \gamma, \beta,\theta) 
&=  \int_0^\infty   f_\sigma(t\vert  \lambda u) \,dP_{\alpha,\beta,\gamma,\theta}(u) 
\label{eq:wML4par} \\
&= \phantom{-} \frac{\Gamma(\beta+\theta)}{\pi} \, {\rm Im} \int_0^\infty e^{-tu} \, 
E^{\gamma+\theta/\alpha}_{\alpha,\beta+\theta}(-\lambda e^{-i\pi\sigma}u^\sigma) \, du  
\label{eq:wML4parPollard} \\
 &= -\frac{1}{\pi} \sum_{k=0}^\infty (-\lambda)^k\sin(\pi\sigma k) \, M_{\alpha,\beta,\gamma,\theta;k} \frac{\Gamma(\sigma k+1)}{t^{\sigma k+1}} 
\label{eq:wML4parPollardseries} 
 \end{align} 
\end{corollary}

 
\begin{corollary}
\label{cor:ML4parconv2}
For $0<\alpha<1$ and $-\theta<\alpha\gamma<\beta$,  the representation  $(\ref{eq:ML4parconv2})$
implies that 
$t^{\beta+\theta-1}E^{\gamma+\theta/\alpha}_{\alpha,\beta+\theta}(-y t^\alpha)$ as a function of $t$  given $y>0$ 
is the Laplace transform of a   density  $q_{\alpha,\beta,\gamma,\theta}(u\vert y)$ 
\begin{align}
y^{\gamma+\theta/\alpha} \, t^{\beta+\theta-1}E^{\gamma+\theta/\alpha}_{\alpha,\beta+\theta}(-y t^\alpha) 
&=  \int_0^\infty  f_\alpha(t\vert  u) \{\rho_{\beta/\alpha-\gamma} \star \rho_{\gamma+\theta/\alpha,y}\}(u) \, du  \nonumber \\
\implies 
t^{\beta+\theta-1}E^{\gamma+\theta/\alpha}_{\alpha,\beta+\theta}(-y t^\alpha)  
&=    \int_0^\infty    e^{-tu} \,  q_{\alpha,\beta,\gamma,\theta}(u\vert y)\, du  \quad (t\ge0)
\label{eq:ML4parconv2LSt} \\
{\rm where}\quad 
q_{\alpha,\beta,\gamma,\theta}(u\vert y)
 &=  \frac{y^{-\gamma-\theta/\alpha}}{\pi}\,{\rm Im}\left \{\widetilde\rho_{\beta/\alpha-\gamma}(e^{-i\pi\alpha}u^\alpha) \,  
 \widetilde\rho_{\gamma+\theta/\alpha,y}(e^{-i\pi\alpha}u^\alpha)\right\} 
 \label{eq:ML3parconv2distribution0} \\
 &= \frac{1}{\pi}\,{\rm Im} \frac{(e^{-i\pi}u)^{\alpha\gamma-\beta}}{(y+  e^{-i\pi\alpha}u^\alpha)^{\gamma+\theta/\alpha}}  
\label{eq:ML4parconv2distribution}
\end{align}
\end{corollary}

We summarise the results of  Theorem~\ref{thm:ML4parfunction} in Table~\ref{table:ML4parDensities}.

\section{An Alternative Representation}
\label{sec:alt}

We now turn to  a representation  of Mittag-Leffler distributions that builds on Corollary~\ref{cor:rhoconvstable},
involving the beta distribution.
By contrast to the foregoing discussion, this approach does not rely on  Mittag-Leffler functions as the point of departure.

\begin{theorem}
\label{thm:altML4pardistribution}
For $0<\alpha<1$ and $-\theta<\alpha\gamma<\beta$,  
the 4-parameter Mittag-Leffler distribution  $P_{\alpha,\beta,\gamma,\theta} \equiv {\rm ML}(\alpha,\beta,\gamma,\theta)$
has a density $p_{\alpha,\beta,\gamma,\theta}(t)$ given by
\begin{align}
p_{\alpha,\beta,\gamma,\theta}(t)
 &= \int_0^1 p_{\alpha,\beta+\theta}(t/u) \, {\rm beta}(u\vert \theta/\alpha+\gamma, \beta/\alpha-\gamma) \frac{du}{u}
 \label{eq:altML4pardistribution}
 \end{align}
 where  $p_{\alpha,\beta+\theta}(t)$ is the density of $P_{\alpha,\beta+\theta}\equiv {\rm ML}(\alpha,\beta+\theta)$.
 Equivalently, if $U$ and $V$ are independent random variables,  
 where $U$  has the beta distribution  ${\rm Beta}(\theta/\alpha+\gamma, \beta/\alpha-\gamma)$ and 
 $V$ has the 2-parameter Mittag-Leffler distribution ${\rm ML}(\alpha,\beta+\theta)$,
 then $T=VU$ has the 4-parameter  Mittag-Leffler distribution  ${\rm ML}(\alpha,\beta,\gamma,\theta)$
 with density~$(\ref{eq:altML4pardistribution})$.
\end{theorem}

\begin{proof}[Proof of Theorem  \ref{thm:altML4pardistribution}]
By (\ref{eq:ML4pardistribution1}) of Corollary~\ref{cor:ML4parconv1}, together  with  Corollary~\ref{cor:rhoconvstable}
\begin{align*}
 p_{\alpha,\beta,\gamma,\theta}(t) 
 &= \Gamma(\beta+\theta) \rho_{\gamma+\theta/\alpha}(t) t^{\beta/\alpha-\gamma}  
  \int_1^\infty  f_\alpha(1 \vert tu) \rho_{\beta/\alpha-\gamma}(u-1)  du \\
   (u\to1/u)\;
 &= \frac{\Gamma(\beta+\theta)}{\Gamma(\gamma+\theta/\alpha)} \frac{t^{(\beta+\theta)/\alpha-1}}{\Gamma(\beta/\alpha-\gamma)}
  \int_0^1  f_\alpha(1 \vert t/u) (1/u-1)^{\beta/\alpha-\gamma-1} \,  \frac{du}{u^2} \\
  &= \frac{\Gamma(\beta+\theta)}{\Gamma(\gamma+\theta/\alpha)} \frac{1}{\Gamma(\beta/\alpha-\gamma)}  \\
  &\qquad \times \int_0^1 (t/u)^{(\beta+\theta)/\alpha-1} f_\alpha(1 \vert t/u) \,  u^{\theta/\alpha+\gamma-1}(1-u)^{\beta/\alpha-\gamma-1} \,  \frac{du}{u} \\
   &= \frac{\Gamma(\beta+\theta)}{\Gamma(\gamma+\theta/\alpha)} \frac{1}{\Gamma(\beta/\alpha-\gamma)}
    \frac{\alpha\Gamma(1+(\beta+\theta)/\alpha)}{\Gamma(1+\beta+\theta)} 
    \frac{\Gamma(\gamma+\theta/\alpha)\Gamma(\beta/\alpha-\gamma)}{\Gamma((\beta+\theta)/\alpha)} \\ 
  &\qquad  \times  \int_0^1 p_{\alpha,\beta+\theta}(t/u) \, {\rm beta}(u\vert \theta/\alpha+\gamma, \beta/\alpha-\gamma) \frac{du}{u} \\
  &= \int_0^1 p_{\alpha,\beta+\theta}(t/u) \, {\rm beta}(u\vert \theta/\alpha+\gamma, \beta/\alpha-\gamma) \frac{du}{u}
 \end{align*}
 with the leading constant evaluating to 1.
 $p_{\alpha,\beta+\theta}(t)$ is read from the definition of $p_{\alpha,\theta}(t)$ in Corollary~\ref{cor:ML2pardistribution}.
 By~(\ref{eq:betaVdensity}),  this is the density of $T=VU$ as described in Theorem~\ref{thm:altML4pardistribution}.
\end{proof}

The particular  case $\beta=\gamma=1$ reproduces  ${\rm ML}(\alpha,\theta)\equiv  {\rm ML}(\alpha,1,1,\theta)$.
This defines  the Mittag-Leffler Markov Chain (MLMC) with  recursion of the form  $T_{n-1}=T_nU_n$
(\cite{Goldschmidt22, GoldschmidtHaas, HoJamesLau, RembartWinkel2018, RembartWinkel2023, Senizergues}).
MLMC  features in models of the growth of random trees.

The representation of Theorem~\ref{thm:altML4pardistribution} is shown in Table~\ref{table:ML4parBetaML}
with various choices of parameters.
In keeping with the rest of the tables,  a  column of corresponding Laplace transforms is reproduced.

Formally, this completes our  narrative  on Mittag-Leffler densities,  mixtures  involving such densities and the stable density,
as well as  the  Laplace transforms of all such objects, featuring Mittag-Leffler functions.

For completenesses, we make reference  to other densities that have arisen in the course of discussion, including  
the generalised positive Linnik density. 
The results summarised here are  well-known but not typically  derived from the explicit convolution perspective  of this paper.

\section{Generalised Positive  Linnik Distribution}
\label{sec:GML}

Table~\ref{table:OtherDensities} presents  densities that have arisen in the course of discussion
(Corollaries~\ref{cor:MLconv2}, \ref{cor:ML3parconv2} and \ref{cor:ML4parconv2})
that are not 
explicitly   in the family  of Mittag-Leffler densities but whose Laplace transforms nonetheless  involve Mittag-Leffler functions.
Unlike all previous instances, these densities are not normalised (with the exception of the first one).

The case $\beta-\alpha\gamma=0$ in Corollary~\ref{cor:ML4parconv2} gives the density whose Laplace transform is the density of the
generalised positive Linnik distribution 
${\rm PL}(\alpha, \gamma+\theta/\alpha,\lambda)$.
We have already encountered the mixture representation  of the latter density
as the special  case $\beta=\alpha\gamma$ in  Theorem~\ref{thm:ML4parfunction} (with $\lambda\equiv y$).
The generalised positive Linnik distribution  ${\rm PL}(\alpha, \gamma+\theta/\alpha,\lambda)$   is summarised   in Table~\ref{table:Linnik}. 

\section{Discussion}
\label{discussion}

The 4-parameter Mittag-Leffler distribution  $P_{\alpha,\beta,\gamma,\theta} \equiv {\rm ML}(\alpha,\beta,\gamma,\theta)$ is the general case,
subsuming all other variants in the literature  as  special cases:
\begin{enumerate}
\item Mittag-Leffler Distribution: $P_\alpha \equiv {\rm ML}(\alpha,1,1,0)$ 
\cite{Feller2, PollardML}
\item 3-parameter Mittag-Leffler Distribution: ${\rm ML}(\alpha,\beta,\gamma) \equiv{\rm ML}(\alpha,\beta,\gamma,0)$ 
\cite{BeghinFoxH, Flajolet, Gorska, Mohle, Tomovski}
\item Beta-Mittag-Leffler Distribution: ${\rm BML(\alpha,\theta,\beta}) \equiv{\rm ML}(\alpha,\beta,0,\theta>0)$  
\cite{Banderier, Flajolet, Mohle, Goldschmidt22}
\item 2-parameter/Generalized Mittag-Leffler Distribution: 
${\rm ML}(\alpha,\theta) \equiv{\rm ML}(\alpha,1,1,\theta>-\alpha)$, equivalently  ${\rm ML}(\alpha,0,0,\theta>0)$ 
\cite{Banderier, BloemReddyOrbanz, Goldschmidt22, GoldschmidtHaas, HoJamesLau, James_Lamperti, James2015, 
JansonProbSurv, Pitman_CSP}
\item Mittag-Leffler Markov Chain defined by    $T_{n-1}=T_nU_n$ where $U$ is a beta variable and $T$ is a 2-parameter/Generalized Mittag-Leffler
variable \cite{Goldschmidt22, GoldschmidtHaas, HoJamesLau, RembartWinkel2018, RembartWinkel2023, Senizergues}

\end{enumerate}

Logically, we might have constructed the general 4-parameter case from the outset. 
We opted instead to start from the single parameter case  for clarity of exposition, despite the inevitable duplication of earlier  results 
that emerge as special cases of later more general formulation.

The primary point though is one of methodology.
All distributions arise from convolutions  and mixtures of gamma and stable densities as the basic construct.
The Laplace-Stietljes transforms of all distributions  thus constructed are  Mittag-Leffler functions
and their compositions with a  Bernstein function of specific form.
There has been no need for explicit complex analysis.
It has sufficed to refer to the representation~(\ref{eq:Pollardinfsum2}) of the stable density,
which is an outcome of  complex analytic inversion  due to Pollard~\cite{Pollard}.

There has also  been no need to appeal to special functions.
For completeness though, we comment   on  how special functions  might be seen to arise from convolutions of gamma densities.

\subsection{Convolutions and Special Functions}
\label{specialfunctions}


We have seen that the  Mittag-Leffler distribution involves the convolution $\{\rho_{\beta-\alpha\gamma}\star  f_\alpha(\cdot\vert  t)\}(1)$. 
By Theorem~\ref{thm:rhoconvstable},  this may be  expressed as 
\begin{align*}
\{\rho_{\beta-\alpha\gamma}\star  f_\alpha(\cdot\vert  t)\}(1)
 &=  t^{\beta/\alpha-\gamma}   \int_0^\infty f_\alpha(1\vert t u) \rho_{\beta/\alpha-\gamma}(u-1) \mathbbm{1}_{[1,\infty)} \, du  \\
  &=  \sum_{k=0}^\infty \frac{(-t)^k}{k!} \frac{1}{\Gamma(\beta-\alpha\gamma-\alpha k)}   
\end{align*}

The  representation~(\ref{eq:Pollardinfsum2}) of the stable density leads to an alternative infinite sum representation 
\begin{align}
 \{\rho_{\beta-\alpha\gamma}\star  f_\alpha(\cdot\vert  t)\}(1)  
 &=  -\frac{1}{\pi}
   \sum_{k=0}^\infty \frac{(-t)^k}{k!} \sin \pi\alpha k  \int_0^\infty \{\rho_{\beta-\alpha\gamma}\star \rho_{1,u}\}(1) \, u^{\alpha k-1} \, du
 \label{eq:rhoconvstableAlt}  
\end{align}

The observation of interest here is that  $\{\rho_{\beta/\alpha-\gamma} \star \rho_{1,u }\}(1) $ is a  
convolution of 2 gamma densities with different scale factors $0$ and $u$.
More generally, the gamma distribution has the well-known  property that it is closed under convolution of densities sharing a common  scale parameter,
{\it i.e.}\ $\{\rho_{\nu,\lambda}\star \rho_{\mu,\lambda}\}(t) = \rho_{\nu+\mu,\lambda}(t)$
with scale  parameter $\lambda$ throughout.
The case of different  scale parameters   is not as simple. 
In particular 
 \begin{align}
 \{\rho_{\nu}\star \rho_{\mu,\lambda}\}(1)
 &= \int_0^1 \rho_{\nu}(1-u) \rho_{\mu,\lambda}(u) du \\
 &=   \frac{\lambda^\mu}{\Gamma(\nu)\Gamma(\mu)} \int_0^1 (1-u )^{\nu-1} \,  u^{\mu-1}e^{-\lambda  u} \, du  \\
 &=   \frac{\lambda^\mu}{\Gamma(\mu+\nu)} \int_0^1{\rm beta}(u\vert \mu,\nu)\, e^{-\lambda  u} \, du
  \label{eq:gammaconv} 
\end{align}
We recognise the integral in~(\ref{eq:gammaconv}) as the Laplace transform   of the beta density, also denoted by 
$M(\mu,\mu+\nu,-\lambda)$, the confluent hypergeometric function (Abramowitz and Stegun~\cite[13.2.1]{AbrSteg}).
Hence, with $\nu=\beta-\alpha\gamma$ and $\mu=1$, (\ref{eq:rhoconvstableAlt}) may be written as 
\begin{align}
 \{\rho_{\beta-\alpha\gamma}\star  f_\alpha(\cdot\vert  t)\}(1)  
 &=  -\frac{1}{\pi}
   \sum_{k=0}^\infty \frac{(-t)^k \sin \pi\alpha k}{k!\Gamma(\beta-\alpha\gamma+1)}   \int_0^\infty u^{\alpha k} \, M(1,\beta-\alpha\gamma+1,-u) \, du
 \label{eq:rhoconvstableAlt1}  
\end{align}
The integral may be looked upon as a Mellin transform of $u M(1,\beta-\alpha\gamma+1,-u)$ defined by 
\begin{align*}
\int_0^\infty u^{\alpha k-1} \, \{u M(1,\beta-\alpha\gamma+1,-u)\} \, du
\end{align*}
While this  illustrates the link between Mittag-Leffler distributions  and  special functions and Mellin transforms,
(\ref{eq:rhoconvstableAlt1})  does not in our view provide additional conceptual insight compared to the expressions
for the convolution $\{\rho_{\beta-\alpha\gamma}\star  f_\alpha(\cdot\vert  t)\}(1)$
given by  Theorem~\ref{thm:rhoconvstable}.
  


 \section{Conclusion}
 \label{sec:conclusion}

We have  used the   convolution and mixing  of gamma and stable densities to construct  densities of higher level distributions.
The convolution between the power density and  the stable density  has been the centrepiece of the discussion.
We  applied the methodology to the construction of Mittag-Leffler distributions, with a 4-parameter instance as the general case subsuming
published instances  as special cases.
The Laplace transforms of the Mittag-Leffler densities are Mittag Leffler functions.
We have shown that the compositions  of the latter with a power law Bernstein function can be constructed  from  corresponding  mixture densities.

\appendix

\section{Mittag-Leffler Densities from Stable Random Variables}
\label{sec:MLrv}

The approach of this paper has been the direct construction of Mittag-Leffler densities  based on convolution and mixing  of gamma and stable densities.
The common approach in the probabilistic  literature is to work with products and powers of stable random variables
to construct Mittag-Leffler densities, as described  in this appendix.

\begin{proposition}
\label{prop:MLfromstablerv}
Let  $X$  be an $\alpha$-stable random variable  with density $f_\alpha(x)\equiv f_\alpha(x\vert 1)$  defined by~$(\ref{eq:stable})$.
Then $T=X^{-\alpha}$ has the Mittag-Leffler distribution  $P_\alpha(t)$ given by
\begin{align}
 dP_\alpha(t) &= \frac{1}{\alpha}  f_\alpha(1 \vert  t)  \,  t^{-1} \, dt
  \equiv    \frac{1}{\alpha}  f_\alpha(t^{-1/\alpha}) \,  t^{-1/\alpha-1} \, dt
  \end{align}
\end{proposition}
\begin{remark}
Proposition~\ref{prop:MLfromstablerv} is well-known.
\end{remark}

\begin{proof}[Proof of Proposition \ref{prop:MLfromstablerv}]
The joint density of $(T=t, X=x)$ is 
$\Pr(t,x) = \Pr(t\vert x)\Pr(x)$ where   $\Pr(x)= f_\alpha(x)$ and $\Pr(t\vert x) = \delta(t-x^{-\alpha})$. 
The marginal density $p_\alpha(t)\equiv \Pr(t)$ of $T=t$  is  
\begin{align*}
p_\alpha(t) &=  \int_0^\infty \delta(t-x^{-\alpha}) f_\alpha(x) \,dx \\
(x\to u^{-1/\alpha}) \;
     &= \frac{1}{\alpha} \int_0^\infty \delta(t-u) f_\alpha(u^{-1/\alpha})  u^{-1/\alpha-1}\,du \\
     &= \frac{1}{\alpha}  f_\alpha(t^{-1/\alpha})\,  t^{-1/\alpha-1} 
\end{align*}
which is the density of the Mittag-Leffler distribution $P_\alpha(t)$.
\end{proof}

\begin{proposition}
\label{prop:stablervratio}
Let $X, Y$ be independent random variables   with   densities   $f_\alpha(x\vert \lambda)$ \\ and $f_\alpha(y)\equiv f_\alpha(y\vert 1)$ respectively.
Then $T=X/Y$ has density $w_{\alpha,\alpha}(t\vert \lambda)$ given by
\begin{align}
 w_{\alpha,\alpha}(t\vert \lambda)  &= \int_0^\infty f_\alpha(t \vert \lambda  u) dP_\alpha(u)
  =  \frac{\sin\pi\alpha}{\pi}\, \frac{\lambda\, t^{\alpha-1}} {\lambda^2+2\lambda \, t^\alpha \cos\pi\alpha+t^{2\alpha}} 
 \label{eq:stablervratio}
 \end{align}
\end{proposition}
\begin{remark}
The second form of Proposition~\ref{prop:stablervratio} for $\lambda =1$  $(X,Y \, {\it iid})$ is well-known.
\end{remark}

\begin{proof}[Proof of Proposition \ref{prop:stablervratio}]
Since $X,Y$ are independent, the joint density of $(T=t, X=x, Y=y)$ is 
$\Pr(t,x,y) = \Pr(t\vert x,y)\Pr(x)\Pr(y)$ where   $\Pr(x)= f_\alpha(x\vert \lambda)$ and $\Pr(t\vert x,y) = \delta(t-x/y)$. 
The marginal density $ w_{\alpha,\alpha}(t\vert \lambda) \equiv \Pr(t)$ of $T=t$  is 
\begin{align*}
 w_{\alpha,\alpha}(t\vert \lambda)  
 &=  \int_0^\infty  \int_0^\infty \delta(t-x/y) f_\alpha(x\vert \lambda)  f_\alpha(y)\,dx dy \\
(x\to yu) \;
 &=   \int_0^\infty \int_0^\infty \delta(t-u) f_\alpha(yu \vert \lambda)  f_\alpha(y) y \,du dy \\
 &=   \int_0^\infty  f_\alpha(yt \vert \lambda) y   f_\alpha(y)\, dy \\
(y\to u^{-1/\alpha}) \;
 &=   \int_0^\infty  f_\alpha(u^{-1/\alpha}t \vert \lambda) u^{-1/\alpha}  \left[\frac{1}{\alpha}f_\alpha(u^{-1/\alpha})u^{-1/\alpha-1} \right]\, du \\
 &=   \int_0^\infty  f_\alpha(t \vert \lambda u)  \, p_\alpha(u) \, du = \int_0^\infty  f_\alpha(t \vert \lambda u)  \, dP_\alpha(u) 
 \end{align*}
 The second form of $w_{\alpha,\alpha}(t\vert \lambda)$ was demonstrated in the proof of Corollary~\ref{cor:MLconv1Composition}.
 \end{proof}

 \begin{proposition}
\label{prop:ChamYorA}
Let $\lambda=1$ in Proposition~\ref{prop:stablervratio} and let $Z=X/Y$, equivalently, $Z=Y/X$ since $X, Y$ now have  the same density.
Then  $T=Z^\alpha$ or $T=Z^{-\alpha}$ has density 
\begin{align}
\Pr(t) &=  \frac{1}{\alpha t} w_{\alpha,\alpha}(1 \vert t) =  \frac{1}{\alpha t}  \int_0^\infty f_\alpha(1 \vert t  u) dP_\alpha(u)
  =  \frac{\sin\pi\alpha}{\pi\alpha}\, \frac{1} {t^2+2t \,  \cos\pi\alpha+1}
 \label{eq:ChamYorA}
 \end{align}
\end{proposition}
\begin{remark}
The second form of Proposition~\ref{prop:ChamYorA}  is in  Chaumont and Yor~\cite[p116(4.21.3)]{ChaumontYor}. 
\end{remark}

\begin{proof}[Proof of Proposition \ref{prop:ChamYorA}]
We prove the case $T=Z^{-\alpha}$. 
Setting $w_{\alpha,\alpha}(z) \equiv w_{\alpha,\alpha}(z\vert 1)$
 \begin{align*}
\Pr(t)  &=  \ \int_0^\infty \delta(t-z^{-\alpha}) w_{\alpha,\alpha}(z)  \, dz   \\
(z\to u^{-1/\alpha}) \;
 &=  \frac{1}{\alpha}\int_0^\infty \delta(t-u) w_{\alpha,\alpha}(u^{-1/\alpha}) u^{-1/\alpha-1} \,du \\
 &=  \frac{1}{\alpha}  w_{\alpha,\alpha}(t^{-1/\alpha}) t^{-1/\alpha-1} \\
 &=    \frac{1}{\alpha t}  \int_0^\infty  f_\alpha(t^{-1/\alpha} \vert  u) t^{-1/\alpha}   \, dP_\alpha(u) \\
 &=    \frac{1}{\alpha t}  \int_0^\infty  f_\alpha(1 \vert t u)   \, dP_\alpha(u) \\
 &=  \frac{1}{\alpha t} w_{\alpha,\alpha}(1 \vert t)
  =  \frac{\sin\pi\alpha}{\pi\alpha}\, \frac{1} {t^2+2t \,  \cos\pi\alpha+1}
 \end{align*}
 Starting with $T=Z^{\alpha}$ gives  $w_{\alpha,\alpha}(1 \vert t^{-1})/\alpha t$, which is  the same  as $w_{\alpha,\alpha}(1 \vert t)/\alpha t$.
 \end{proof}


 \bibliography{MittagLeffler.bib}{}

\begin{thebibliography}{10}

\bibitem{AbrSteg}
M.~Abramowitz and I.~Stegun.
\newblock {\em Handbook of Mathematical Functions}.
\newblock Dover, New York, 1965.

\bibitem{Banderier}
Cyril Banderier, Markus Kuba, and Michael Wallner.
\newblock Phase transitions of composition schemes: {Mittag-Leffler and mixed
  Poisson distributions}.
\newblock \href{https://arxiv.org/abs/2103.03751}{arXiv:2103.03751}, 2021.

\bibitem{BarabasiAlbert}
Albert-L\'{a}szl\'{o} Barab\'{a}si and R\'{e}ka Albert.
\newblock Emergence of scaling in random networks.
\newblock {\em Science}, 286(5439):509--512, 1999.

\bibitem{Barabesi}
L.~Barabesi, A.~Cerasa, A.~Cerioli, and D.~Perrotta.
\newblock A new family of tempered distributions.
\newblock {\em Electronic Journal of Statistics}, 10(2):3871--3893, 2016.

\bibitem{BeghinFoxH}
L.~Beghin, L.~Cristofaro, and J.~L.~Da Silva.
\newblock {Fox-$H$} densities and completely monotone generalized {W}right
  functions.
\newblock \href{https://arxiv.org/abs/2310.01948}{arXiv:2310.01948}, 2023.

\bibitem{Bercu2022}
Bernard Bercu.
\newblock On the elephant random walk with stops playing hide and seek with the
  {M}ittag--{L}effler distribution.
\newblock {\em Journal of Statistical Physics}, 189(1):12, Aug 2022.

\bibitem{Bertoin}
J.~Bertoin, T.~Fujita, B.~Roynette, and M.~Yor.
\newblock {On a particular class of self-decomposable random variables: the
  durations of Bessel excursions straddling independent exponential times}.
\newblock {\em Probab. Math. Statist.}, 26(2):315–366, 2006.

\bibitem{BertoinYor}
Jean Bertoin and Marc Yor.
\newblock {On Subordinators, Self-Similar Markov Processes and Some
  Factorizations of the Exponential Variable}.
\newblock {\em Electronic Communications in Probability}, 6(none):95 -- 106,
  2001.

\bibitem{BloemReddyOrbanz}
Benjamin Bloem-Reddy and Peter Orbanz.
\newblock {Random-Walk Models of Network Formation and Sequential Monte Carlo
  Methods for Graphs}.
\newblock {\em Journal of the Royal Statistical Society Series B: Statistical
  Methodology}, 80(5):871--898, 08 2018.

\bibitem{Bondesson}
Lennart Bondesson.
\newblock {\em Generalized Gamma Convolutions and Related Classes of
  Distributions and Densities}.
\newblock Lecture Notes in Statistics, 76. Springer-Verlag, New York, 1992.

\bibitem{ChaumontYor}
L.~Chaumont and M.~Yor.
\newblock {\em Exercises in Probability: A Guided Tour from Measure Theory to
  Random Processes, via Conditioning}.
\newblock Cambridge Series in Statistical and Probabilistic Mathematics.
  Cambridge University Press, 2012.

\bibitem{Oliveira}
E.~Capelas de~Oliveira, F.~Mainardi, and J.~Vaz.
\newblock {Models based on Mittag-Leffler functions for anomalous relaxation in
  dielectrics}.
\newblock {\em The European Physical Journal Special Topics}, 193(1):161--171,
  Mar 2011.

\bibitem{Feller2}
William Feller.
\newblock {\em An Introduction to Probability Theory and its Applications, Vol.
  II}.
\newblock Wiley, New York, 1971.

\bibitem{Flajolet}
Philippe Flajolet, Philippe Dumas, and Vincent Puyhaubert.
\newblock {Some exactly solvable models of urn process theory}.
\newblock {\em {Discrete Mathematics \& Theoretical Computer Science}}, {DMTCS
  Proceedings vol. AG, Fourth Colloquium on Mathematics and Computer Science
  Algorithms, Trees, Combinatorics and Probabilities}, January 2006.

\bibitem{Fox}
Charles Fox.
\newblock {The $G$ and $H$ functions as symmetrical {F}ourier kernels}.
\newblock {\em Trans. Amer. Math. Soc.}, 98:395--429, 1961.

\bibitem{Fujii}
Keisuke Fujii.
\newblock Collisional-energy-cascade model for nonthermal velocity
  distributions of neutral atoms in plasmas.
\newblock {\em Phys. Rev. E}, 108:025204, Aug 2023.

\bibitem{Giusti}
Andrea Giusti, Ivano Colombaro, Roberto Garra, Roberto Garrappa, Federico
  Polito, Marina Popolizio, and Francesco Mainardi.
\newblock {A practical guide to Prabhakar fractional calculus}.
\newblock {\em Fractional Calculus and Applied Analysis}, 23(1):9--54, 2020.

\bibitem{Goldschmidt22}
Christina Goldschmidt, B\'en\'edicte Haas, and Delphin S\'enizergues.
\newblock Stable graphs: distributions and line-breaking construction.
\newblock {\em Annales Henri Lebesgue}, 5:841--904, 2022.

\bibitem{GoldschmidtHaas}
Christina Goldschmidt and Bénédicte Haas.
\newblock {A line-breaking construction of the stable trees}.
\newblock {\em Electronic Journal of Probability}, 20:1 -- 24, 2015.

\bibitem{gorenflo2014mittag}
Rudolf Gorenflo, Anatoly~A Kilbas, Francesco Mainardi, and Sergei~V Rogosin.
\newblock {\em Mittag-Leffler functions, related topics and applications}.
\newblock Springer, 2014.

\bibitem{Gorska}
K.~G\'{o}rska, Andrzej Horzela, Ambra Lattanzi, and Tibor Pog\'{a}ny.
\newblock {On complete monotonicity of three parameter Mittag-Leffler
  function}.
\newblock {\em Applicable Analysis and Discrete Mathematics}, 15:118--128, 04
  2021.

\bibitem{Gripenberg}
G.~Gripenberg, S.~O. Londen, and O.~Staffans.
\newblock {\em Volterra Integral and Functional Equations}.
\newblock Encyclopedia of Mathematics and its Applications. Cambridge
  University Press, 1990.

\bibitem{Haubold}
Hans~J. Haubold, Arak~M. Mathai, and Ram~K. Saxena.
\newblock {M}ittag-{L}effler functions and their applications.
\newblock {\em J. Appl. Math.}, 2011:298628:1--298628:51, 2011.

\bibitem{HoJamesLau}
Man-Wai Ho, Lancelot~F. James, and John~W. Lau.
\newblock {Gibbs partitions, Riemann–Liouville fractional operators,
  Mittag–Leffler functions, and fragmentations derived from stable
  subordinators}.
\newblock {\em Journal of Applied Probability}, 58(2):314–334, 2021.

\bibitem{James_Lamperti}
Lancelot~F. James.
\newblock Lamperti-type laws.
\newblock {\em Ann. Appl. Probab.}, 20(4):1303--1340, 2010.

\bibitem{James2015}
Lancelot~F James.
\newblock Generalized {Mittag-Leffler} distributions arising as limits in
  preferential attachment models.
\newblock \href{https://arxiv.org/abs/1509.07150}{arXiv:1509.07150}, September
  2015.

\bibitem{JamesRoynetteYor}
Lancelot~F. James, Bernard Roynette, and Marc Yor.
\newblock {Generalized Gamma Convolutions, Dirichlet means, Thorin measures,
  with explicit examples}.
\newblock {\em Probability Surveys}, 5(none):346 -- 415, 2008.

\bibitem{Janson}
Svante Janson.
\newblock Limit theorems for triangular urn schemes.
\newblock {\em Probability Theory and Related Fields}, 134(3):417--452, Mar
  2006.

\bibitem{JansonProbSurvA}
Svante Janson.
\newblock {Addendum to Moments of Gamma type and the Brownian supremum process
  area}.
\newblock {\em Probability Surveys}, 7:207 -- 208, 2010.

\bibitem{JansonProbSurv}
Svante Janson.
\newblock {Moments of Gamma type and the Brownian supremum process area}.
\newblock {\em Probability Surveys}, 7:1 -- 52, 2010.

\bibitem{Janson2024}
Svante Janson.
\newblock Almost sure and moment convergence for triangular {P}\'olya urns.
\newblock \href{https://arxiv.org/abs/2402.01299}{arXiv:2402.01299}, February
  2024.

\bibitem{Jedidi}
Wissem Jedidi, Thomas Simon, and Min Wang.
\newblock Density solutions to a class of integro-differential equations.
\newblock {\em Journal of Mathematical Analysis and Applications},
  458(1):134--152, February 2018.

\bibitem{Jose}
Kanichukattu~Korakutty Jose, Padmini Uma, Vanaja~Seetha Lekshmi, and
  Hans~Joachim Haubold.
\newblock Generalized {M}ittag-{L}effler distributions and processes for
  applications in astrophysics and time series modeling.
\newblock In Hans~J. Haubold and A.M. Mathai, editors, {\em Proceedings of the
  Third UN/ESA/NASA Workshop on the International Heliophysical Year 2007 and
  Basic Space Science}, pages 79--92, Berlin, Heidelberg, 2010. Springer Berlin
  Heidelberg.

\bibitem{kiryakova1993generalized}
V.S. Kiryakova.
\newblock {\em Generalized Fractional Calculus and Applications}.
\newblock Chapman \& Hall/CRC Research Notes in Mathematics Series. Taylor \&
  Francis, 1993.

\bibitem{Korolev2020}
V.~Korolev, A.~Gorshenin, and A.~Zeifman.
\newblock On mixture representations for the {Generalized Linnik} distribution
  and their applications in limit theorems.
\newblock {\em Journal of Mathematical Sciences}, 246(4):503--518, Apr 2020.

\bibitem{Korolev2016}
V.~Yu. Korolev and A.~I. Zeifman.
\newblock A note on mixture representations for the {Linnik and Mittag-Leffler}
  distributions and their applications*.
\newblock {\em Journal of Mathematical Sciences}, 218(3):314--327, Oct 2016.

\bibitem{Lamperti}
John~W. Lamperti.
\newblock An occupation time theorem for a class of stochastic processes.
\newblock {\em Transactions of the American Mathematical Society}, 88:380--387,
  1958.

\bibitem{MainardiGarrappa}
Francesco Mainardi and Roberto Garrappa.
\newblock {On complete monotonicity of the Prabhakar function and non-Debye
  relaxation in dielectrics}.
\newblock {\em Journal of Computational Physics}, 293:70--80, 2015.
\newblock Fractional PDEs.

\bibitem{Mathai}
A.M. Mathai, R.K. Saxena, and H.J. Haubold.
\newblock {\em The H-Function: Theory and Applications}.
\newblock Springer New York, 2009.

\bibitem{miller1993introduction}
K.S. Miller and B.~Ross.
\newblock {\em An Introduction to the Fractional Calculus and Fractional
  Differential Equations}.
\newblock Wiley, 1993.

\bibitem{MiyazakiTakei}
Tatsuya Miyazaki and Masato Takei.
\newblock Limit theorems for the ‘laziest’ minimal random walk model of
  elephant type.
\newblock {\em Journal of Statistical Physics}, 6 2020.

\bibitem{Mohle}
Martin Möhle.
\newblock {A restaurant process with cocktail bar and relations to the
  three-parameter Mittag–Leffler distribution}.
\newblock {\em Journal of Applied Probability}, 58(4):978–1006, 2021.

\bibitem{Pakes95}
Anthony~G. Pakes.
\newblock {Characterization of discrete laws via mixed sums and Markov
  branching processes}.
\newblock {\em Stochastic Processes and their Applications}, 55(2):285--300,
  1995.

\bibitem{Pakes}
Anthony~G. Pakes.
\newblock On generalized stable and related laws.
\newblock {\em Journal of Mathematical Analysis and Applications},
  411(1):201--222, 2014.

\bibitem{PensonGorska}
Karol Penson and K.~G\'{o}rska.
\newblock Exact and explicit probability densities for one-sided {L}\'{e}vy
  stable distributions.
\newblock {\em Physical Review Letters}, 105:210604, 11 2010.

\bibitem{Pillai}
R.~N. Pillai.
\newblock {On Mittag-Leffler functions and related distributions}.
\newblock {\em Annals of the Institute of Statistical Mathematics},
  42(1):157--161, Mar 1990.

\bibitem{Pitman_CSP}
J.~Pitman.
\newblock {\em {Combinatorial Stochastic Processes}}, volume 1875 of {\em
  Lecture Notes in Mathematics}.
\newblock Springer-Verlag, Berlin, 2006.
\newblock Lectures from the $32^{\rm nd}$ Summer School on Probability Theory
  held in Saint-Flour, July 7--24, 2002, With a foreword by Jean Picard.

\bibitem{PitmanYor}
Jim Pitman and Marc Yor.
\newblock {The two-parameter Poisson-Dirichlet distribution derived from a
  stable subordinator}.
\newblock {\em The Annals of Probability}, 25(2):855 -- 900, 1997.

\bibitem{Pollard}
Harry Pollard.
\newblock {The representation of $e^{ - x^\lambda }$ as a Laplace integral}.
\newblock {\em Bulletin of the American Mathematical Society}, 52(10):908 --
  910, 1946.

\bibitem{PollardML}
Harry Pollard.
\newblock {The completely monotonic character of the Mittag-Leffler function
  $E_a \left( { - x} \right)$}.
\newblock {\em Bulletin of the American Mathematical Society}, 54(12):1115 --
  1116, 1948.

\bibitem{RembartWinkel2018}
Franz Rembart and Matthias Winkel.
\newblock {Recursive construction of continuum random trees}.
\newblock {\em The Annals of Probability}, 46(5):2715 -- 2748, 2018.

\bibitem{RembartWinkel2023}
Franz Rembart and Matthias Winkel.
\newblock A binary embedding of the stable line-breaking construction.
\newblock {\em Stochastic Processes and their Applications}, 163:424--472,
  2023.

\bibitem{samko1993fractional}
S.~Samko, A.A. Kilbas, and O.~Marichev.
\newblock {\em Fractional Integrals and Derivatives}.
\newblock Taylor \& Francis, 1993.

\bibitem{Sato}
K.~Sato.
\newblock {\em L\'{e}vy Processes and Infinitely Divisible Distributions}.
\newblock Cambridge University Press, Cambridge, 1999.

\bibitem{Schilling}
René~L. Schilling, Renming Song, and Zoran Vondracek.
\newblock {\em Bernstein Functions: Theory and Applications}.
\newblock De Gruyter, Berlin, Boston, 2012.

\bibitem{Schneider}
W.~R. Schneider.
\newblock {\em Generalized Levy Distributions as Limit Laws}, pages 279--304.
\newblock Springer Netherlands, Dordrecht, 1988.

\bibitem{Senizergues}
Delphin S{\'e}nizergues.
\newblock {Geometry of weighted recursive and affine preferential attachment
  trees}.
\newblock {\em Electronic Journal of Probability}, 26:1 -- 56, 2021.

\bibitem{Shanbhag}
D.~N. Shanbhag and M.~Sreehari.
\newblock On certain self-decomposable distributions.
\newblock {\em Zeitschrift f{\"u}r Wahrscheinlichkeitstheorie und Verwandte
  Gebiete}, 38(3):217--222, Sep 1977.

\bibitem{SteutelvanHarn}
F.W. Steutel and K.~van Harn.
\newblock {\em Infinite Divisibility of Probability Distributions on the Real
  Line}.
\newblock Marcel Dekker, New York, 2003.

\bibitem{Titchmarsh}
E.C. Titchmarsh.
\newblock {\em {Introduction to the Theory of Fourier Integrals}}.
\newblock Clarendon Press, 1948.

\bibitem{Tomovski}
Zivorad Tomovski, Tibor Pogány, and Hari Srivastava.
\newblock {Laplace type integral expressions for a certain three-parameter
  family of generalized Mittag-Leffler functions with applications}.
\newblock {\em Journal of the Franklin Institute}, 351, 12 2014.

\bibitem{Widder}
David~Vernon Widder.
\newblock {\em Laplace transform (PMS-6)}.
\newblock Princeton university press, 2015.

\end{thebibliography}
\bibliographystyle{plain} 

\newpage

%
\begin{table}[h!]
{\normalsize
\begin{equation*}
\begin{array}{| c | c  | c | } 
\cline{2-3} 
\multicolumn{1}{c|}{} & 
\multicolumn{1}{ c |}{\cellcolor{lightgray} \textrm{Probability Density}} & 
 \multicolumn{1}{c |} {\cellcolor{lightgray}
\begin{array}{c}
\textrm{Laplace Transform} \\
\textrm{of density} 
\end{array}} 
\TBstrut\\
\hline 
\hline 
& \multicolumn{1}{l |} {\textrm{\footnotesize Corollary~\ref{cor:MLconv1}: Mittag-Leffler density }} & 
 \multicolumn{1}{l |} {\textrm{\footnotesize LT of ML density}} \\
& \multicolumn{1}{l |} {\textrm{\footnotesize Parameters:  $0<\alpha<1$ \, }} & \\
 \textrm{\scriptsize 1}  & 
 \begin{aligned}
 p_{\alpha}(t) &=  \{\rho_{1-\alpha}\star  f_\alpha(\cdot\vert  t)\}(1) \Tstrut\\
 &= \frac{1}{\alpha}  f_\alpha(1 \vert  t)  \,  t^{-1} \equiv \frac{1}{\alpha}  f_\alpha(t^{-1/\alpha}) \,  t^{-1/\alpha-1} \Tstrut\\
 &=  \frac{1}{\pi\alpha t}\,{\rm Im}  \int_0^\infty e^{-u} \, e^{-te^{-i\pi\alpha}u^\alpha}  \, du  \Tstrut\\
 &= \frac{1}{\pi\alpha}  \sum_{k=0}^\infty \frac{(-t)^{k-1}}{k!} \sin(\pi\alpha k) \, \Gamma(\alpha k+1)  \TBstrut\\
 \end{aligned} 
 &
E_{\alpha}(-x) \TBstrut\\
\hline 
\hline 
&  \multicolumn{1}{l |} {\textrm{\footnotesize Corollary~\ref{cor:MLconv1Composition}: $\{\alpha, \sigma\}$ Mixing}} &  
 \multicolumn{1}{l |} {\textrm{\footnotesize Composition with $\lambda x^\sigma$ }} \\
& \multicolumn{1}{l |} {\textrm{\footnotesize Parameters: $0<\alpha<1, 0<\sigma<1; \lambda>0$}}  &   \\
\textrm{\scriptsize 2} & 
 \multicolumn{1}{c |} {
\begin{array}{c}
w_{\sigma,\alpha}(t \vert \lambda) = 
\displaystyle \int_0^\infty  f_\sigma(t\vert \lambda u) \, dP_{\alpha}(u) \TBstrut\\
\begin{aligned}
&= \phantom{-} \dfrac{1}{\pi} {\rm Im} \displaystyle\int_0^\infty e^{-tu} \, E_\alpha(-\lambda e^{-i\pi\sigma}u^\sigma) \, du   \Bstrut\\
 &= -\dfrac{1}{\pi}\
\displaystyle \sum_{k=0}^\infty (-\lambda)^k\sin(\pi\sigma k) \, \frac{\Gamma(\sigma k+1)}{\Gamma(\alpha k+1)} \, t^{-\sigma k-1}  \Bstrut\\
 \end{aligned}
\end{array}} &
E_{\alpha}(-\lambda x^\sigma)  
\TBstrut\\
\hline 
& \multicolumn{1}{l |} {\textrm{\footnotesize Corollary~\ref{cor:MLconv1Composition}: $\{\alpha,\alpha\}$ Mixing}} & 
 \multicolumn{1}{l |} {\textrm{\footnotesize Composition with $\lambda x^\alpha$ }} \\
& \multicolumn{1}{l |} {\textrm{\footnotesize Parameters: $0<\alpha<1; \lambda>0$}}  &   \\
\textrm{\scriptsize 2a}  & 
 \multicolumn{1}{c |} {
\begin{array}{c}
w_{\alpha,\alpha}(t \vert \lambda) = 
\displaystyle \int_0^\infty  f_\alpha(t\vert \lambda u) \, dP_\alpha(u)  \TBstrut\\
  = \dfrac{\sin\pi\alpha}{\pi}\, \dfrac{\lambda\, t^{\alpha-1}} {\lambda^2+2\lambda  t^\alpha \cos\pi\alpha+t^{2\alpha}}   \TBstrut\\
 \end{array}} &
E_{\alpha}(-\lambda x^\alpha) 
\TBstrut\\ 
 \hline 
\multicolumn{3}{l}{\footnotesize \textrm{Notes:}} \\
\multicolumn{3}{l}{\textrm{\footnotesize  $f_\alpha(t\vert u)$  is the  stable density with Laplace transform $e^{-ux^\alpha}$ \, $(u>0)$ }} \\
\multicolumn{3}{l}{ \textrm{\footnotesize  $p_{\alpha}(t)$  is the density of the Mittag-Leffler distribution $P_{\alpha}(t)$  } } \\
\multicolumn{3}{l}{\textrm{\footnotesize  $\rho_{\gamma}(t)$  is the  unnormalised power density with Laplace transform 
$\widetilde\rho_{\gamma}(x)=x^{-\gamma}$ } } \\
\end{array} \\
\end{equation*}}
\caption{Theorem~\ref{thm:MLfunction}: Mittag-Leffler densities and their Laplace transforms}
\label{table:MLDensities}
\end{table}

\newpage

%
\begin{table}[h!]
{\normalsize
\begin{equation*}
\begin{array}{| c | c  | c | } 
\cline{2-3} 
\multicolumn{1}{c|}{} & 
\multicolumn{1}{ c |}{\cellcolor{lightgray} \textrm{Probability Density}} & 
 \multicolumn{1}{c |} {\cellcolor{lightgray}
\begin{array}{c}
\textrm{Laplace Transform} \\
\textrm{of density} 
\end{array}} 
\TBstrut\\
\hline 
\hline 
& \multicolumn{1}{l |} {\textrm{\footnotesize Corollary~\ref{cor:ML3parconv1}: 3-parameter Mittag-Leffler density }} & 
 \multicolumn{1}{l |} {\textrm{\footnotesize LT of 3-par ML density}} \\
& \multicolumn{1}{l |} {\textrm{\footnotesize $0<\alpha<1$, $0<\alpha\gamma<\beta$ }} & \\
 \textrm{\scriptsize 1}  & 
 \begin{aligned}
 p_{\alpha,\beta,\gamma}(t) &=  \Gamma(\beta) \, \rho_\gamma(t) \{\rho_{\beta-\alpha\gamma}\star  f_\alpha(\cdot\vert  t)\}(1)  \\
 &= \Gamma(\beta) \rho_\gamma(t) t^{\beta/\alpha-\gamma}  \;\times\\
 & \qquad   \displaystyle \int_0^\infty  f_\alpha(1 \vert tu) \rho_{\beta/\alpha-\gamma}(u-1)\mathbbm{1}_{[1,\infty)}du  \\ 
 &= \frac{\Gamma(\beta)  \rho_\gamma(t)}{\pi}{\rm Im}  \int_0^\infty e^{-u} \, (e^{-i\pi}u)^{\alpha\gamma-\beta} \, e^{-t (e^{-i\pi}u)^\alpha} du  \\
 &= \frac{\Gamma(\beta)}{\Gamma(\gamma)} \,  
      \sum_{k=0}^\infty \frac{(-1)^k}{k!} \frac{ t^{\gamma+k-1}}{\Gamma(\beta-\alpha\gamma-\alpha k)}   \Bstrut\\
 \end{aligned}  &
\Gamma(\beta) E^\gamma_{\alpha,\beta}(-x) \TBstrut\\
\hline 
\hline 
&  \multicolumn{1}{l |} {\textrm{\footnotesize Corollary~\ref{cor:ML3parconv1Composition}: 
$\{\alpha, \sigma\}$ Mixing}} &  
 \multicolumn{1}{l |} {\textrm{\footnotesize Composition with $\lambda x^\sigma$ }} \\
 & \multicolumn{1}{l |} {\textrm{\footnotesize $0<\alpha<1, 0<\sigma<1, 0<\alpha\gamma<\beta; \,\lambda>0$}} & \\
\textrm{\scriptsize 2} & 
 \multicolumn{1}{c |} {
\begin{array}{c}
\displaystyle \int_0^\infty  f_\sigma(t\vert \lambda u) \, dP_{\alpha,\beta,\gamma}(u) \TBstrut\\
\begin{aligned}
&= \phantom{-} \dfrac{\Gamma(\beta)}{\pi} {\rm Im} \displaystyle\int_0^\infty e^{-tu} \, 
      E^\gamma_{\alpha,\beta}(-\lambda e^{-i\pi\sigma}u^\sigma) \, du   \Bstrut\\
 &= -\dfrac{1}{\pi}\
\displaystyle \sum_{k=0}^\infty (-\lambda)^k\sin(\pi\sigma k) \, M_{\alpha,\beta,\gamma;k} \frac{\Gamma(\sigma k+1)}{t^{\sigma k+1}}  \TBstrut\\
 \end{aligned}
\end{array}} &
\Gamma(\beta) E^\gamma_{\alpha,\beta}(-\lambda x^\sigma)  
\TBstrut\\
\hline 
&  \multicolumn{1}{l |} {\textrm{\footnotesize Corollary~\ref{cor:ML3parconv1Composition}: $\{\alpha, \alpha\}$ Mixing (as above, $\sigma=\alpha$)}} &  
 \multicolumn{1}{l |} {\textrm{\footnotesize Composition with $\lambda x^\alpha$ }} \\
 & \multicolumn{1}{l |} {\textrm{\footnotesize $0<\alpha<1, 0<\alpha\gamma<\beta; \,\lambda>0$}} & \\
\textrm{\scriptsize 2a} & 
\displaystyle \int_0^\infty  f_\alpha(t\vert \lambda u) \, dP_{\alpha,\beta,\gamma}(u) &
\Gamma(\beta) E^\gamma_{\alpha,\beta}(-\lambda x^\alpha)  \TBstrut\\
\hline 
\multicolumn{3}{l}{\textrm{\footnotesize  Notes: ($f_\alpha(t\vert u), \rho_{\gamma}(t)$  are as in Table~\ref{table:MLDensities})}} \\
\multicolumn{3}{l}{ \textrm{\footnotesize  $p_{\alpha,\beta,\gamma}(t)$  is the density of the 3-parameter 
        Mittag-Leffler distribution $P_{\alpha,\beta,\gamma}(t)$  } } \\
 \multicolumn{3}{l}{\textrm{\footnotesize $M_{\alpha,\beta,\gamma;k}
  = \dfrac{\Gamma(\beta)\Gamma(\gamma+k)}{\Gamma(\gamma)\,\Gamma(\beta+\alpha k)}$
 is the  $k^{\rm th}$ moment of $P_{\alpha,\beta,\gamma}(t)$} 
}  \Tstrut\\
\end{array} \\
\end{equation*}}
\caption{Theorem~\ref{thm:ML3parfunction}: 3-parameter Mittag-Leffler densities and their Laplace transforms}
\label{table:ML3parDensities}
\end{table}

\newpage
%
\begin{table}[h!]
{\normalsize
\begin{equation*}
\begin{array}{| c | c  | c | } 
\cline{2-3} 
\multicolumn{1}{c|}{} & 
\multicolumn{1}{ c |}{\textrm{Probability Density}} & 
 \multicolumn{1}{c |} {
\begin{array}{c}
\textrm{Laplace Transform} \\
\textrm{of density} 
\end{array}} 
\TBstrut\\
\hline 
\hline 
& \multicolumn{1}{l |} {\textrm{\footnotesize Corollary~\ref{cor:ML4parconv1}: 4-parameter Mittag-Leffler density }} & 
 \multicolumn{1}{l |} {\textrm{\footnotesize LT of 4-par ML density}} \\
& \multicolumn{1}{l |} {\textrm{\footnotesize $0<\alpha<1$, $-\theta<\alpha\gamma<\beta$ }} & \\
 \textrm{\scriptsize 1}  & 
\begin{aligned}
 p_{\alpha,\beta,\gamma,\theta}(t) &=  \Gamma(\beta+\theta) \, \rho_{\gamma+\theta/\alpha}(t) \{\rho_{\beta-\alpha\gamma}\star  f_\alpha(\cdot\vert  t)\}(1) \\
 &= \Gamma(\beta+\theta) \rho_{\gamma+\theta/\alpha}(t) t^{\beta/\alpha-\gamma} \; \times \\
 & \qquad  \displaystyle \int_0^\infty  f_\alpha(1 \vert tu) \rho_{\beta/\alpha-\gamma}(u-1)\mathbbm{1}_{[1,\infty)}du  \\ 
 &= \Gamma(\beta+\theta) \rho_{\gamma+\theta/\alpha}(t) \;\times \\
           &  \qquad \frac{1}{\pi} {\rm Im}  \int_0^\infty e^{-u}  (e^{-i\pi}u)^{\alpha\gamma-\beta} e^{-t (e^{-i\pi}u)^\alpha} du \\
 &= \frac{\Gamma(\beta+\theta)}{\Gamma(\gamma+\theta/\alpha)} 
      \sum_{k=0}^\infty \frac{(-1)^k}{k!} \frac{t^{\gamma+\theta/\alpha+k-1} }{\Gamma(\beta-\alpha\gamma-\alpha k)}   \Bstrut\\
 \end{aligned}  &
\Gamma(\beta+\theta) E^{\gamma+\theta/\alpha}_{\alpha,\beta+\theta}(-x) \TBstrut\\
\hline
\hline 
&  \multicolumn{1}{l |} {\textrm{\footnotesize Corollary~\ref{cor:ML2pardistribution}: $\beta-\alpha\gamma=1-\alpha$, $\theta>-\alpha\gamma$ }} 
& \multicolumn{1}{l |} {\textrm{\footnotesize $\gamma=(\beta-1)/\alpha+1$}}  \\
 \textrm{\scriptsize 2}  & 
p_{\alpha,\beta,\gamma,\theta}(t)  =   \Gamma(\beta+\theta) \,  \rho_{\gamma+\theta/\alpha}(t) \, p_\alpha(t) &
\Gamma(\beta+\theta) E^{\gamma+\theta/\alpha}_{\alpha,\beta+\theta}(-x) \Bstrut\\
\hline 
&  \multicolumn{1}{l |} {\textrm{\footnotesize Corollary~\ref{cor:ML2pardistribution}: $\beta=\gamma=1$, $\theta>-\alpha$: 
   $P_{\alpha,\theta}\equiv {\rm ML}(\alpha,\theta)$ }} &  \\
 \textrm{\scriptsize 2a}  & 
 \multicolumn{1}{c |} {
\begin{array}{c}
p_{\alpha,\theta}(t)  
 = \frac{\Gamma(1+\theta)}{\Gamma(1+\theta/\alpha)} \, t^{\theta/\alpha} \, p_\alpha(t) \Tstrut\\
\end{array}} &
\Gamma(1+\theta) E^{1+\theta/\alpha}_{\alpha,1+\theta}(-x) \Bstrut\\
\hline 
\hline 
&  \multicolumn{1}{l |} {\textrm{\footnotesize Corollary~\ref{cor:ML4parconv1Composition}: $\{\alpha, \sigma\}$ Mixing}} &  
 \multicolumn{1}{l |} {\textrm{\footnotesize Composition with $\lambda x^\sigma$ }} \\
 & \multicolumn{1}{l |} {\textrm{\footnotesize $0<\alpha<1, 0<\sigma<1, -\theta<\alpha\gamma<\beta; \,\lambda>0$}} & \\
\textrm{\scriptsize 3} & 
 \multicolumn{1}{c |} {
\begin{array}{c}
\displaystyle \int_0^\infty  f_\sigma(t\vert \lambda u) \, dP_{\alpha,\beta,\gamma,\theta}(u) \TBstrut\\
\begin{aligned}
&= \dfrac{\Gamma(\beta+\theta)}{\pi} {\rm Im} \displaystyle\int_0^\infty e^{-tu} \, 
      E^{\gamma+\theta/\alpha}_{\alpha,\beta+\theta}(-\lambda e^{-i\pi\sigma}u^\sigma) \, du   \Bstrut\\
 &= -\dfrac{1}{\pi}\
\displaystyle \sum_{k=0}^\infty (-\lambda)^k\sin(\pi\sigma k) \, M_{\alpha,\beta,\gamma,\theta;k} \frac{\Gamma(\sigma k+1)}{t^{\sigma k+1}}  \Bstrut\\
 \end{aligned}
\end{array}} &
\Gamma(\beta+\theta) E^{\gamma+\theta/\alpha}_{\alpha,\beta+\theta}(-\lambda x^\sigma)  
\TBstrut\\
\hline 
\multicolumn{3}{l}{\footnotesize  \textrm{Notes: {\footnotesize  ($f_\alpha(t\vert u), \rho_{\gamma}(t)$  are as in Table~\ref{table:MLDensities})}   }} \\
\multicolumn{3}{l}{ \textrm{\footnotesize  $p_{\alpha,\beta,\gamma,\theta}(t)$  is the density 
and $M_{\alpha,\beta,\gamma,\theta;k}$ the  $k^{\rm th}$ moment
of the 4-parameter distribution $P_{\alpha,\beta,\gamma,\theta}(t)$  } } \\
\end{array} \\
\end{equation*}}
\caption{Theorem~\ref{thm:ML4parfunction}: 4-parameter Mittag-Leffler densities and their Laplace transforms.}
\label{table:ML4parDensities}
\end{table}

\newpage
%
\begin{table}[h!]
{\normalsize
\begin{equation*}
\begin{array}{| c  | c | c || c |} 
\hline
\cellcolor{lightgray} \textrm{\footnotesize Distribution of $U$} & 
\cellcolor{lightgray} \textrm{\footnotesize Distribution of $V$} & 
\cellcolor{lightgray} \textrm{\footnotesize Distribution of $T=VU$}  
& \cellcolor{lightgray} \textrm{\footnotesize LT of density of $T$}
\TBstrut\\
\hline
\multicolumn{4}{c}{} \\
\multicolumn{4}{l} {\textrm{\footnotesize General Case of Theorem~\ref{thm:altML4pardistribution}: $0<\alpha<1$, $-\theta<\alpha\gamma<\beta$ }}  \strut\\
\hline
{\rm Beta}\left(\tfrac{\theta}{\alpha}+\gamma,\tfrac{\beta}{\alpha}-\gamma\right) & 
{\rm ML}(\alpha, \beta+\theta) & 
^\dagger{\rm ML}(\alpha, \beta, \gamma, \theta) 
& \Gamma(\beta+\theta) E^{\gamma+\theta/\alpha}_{\alpha,\beta+\theta}(-x) 
\TBstrut\\
\hline
\multicolumn{4}{c}{} \\
\multicolumn{4}{l} {\textrm{\footnotesize 3-parameter Mittag-Leffler  case \cite{BertoinYor, HoJamesLau, Mohle}: 
$0<\alpha<1$, $\theta=0, 0<\alpha\gamma<\beta$ }}  \strut\\
\hline
{\rm Beta}\left(\gamma,\tfrac{\beta}{\alpha}-\gamma\right) & 
{\rm ML}(\alpha, \beta) & 
{\rm ML}(\alpha, \beta,\gamma) 
& \Gamma(\beta) E^{\gamma}_{\alpha,\beta}(-x)
\TBstrut\\
\hline
\multicolumn{4}{c}{} \\
\multicolumn{4}{l} {\textrm{\footnotesize  beta-Mittag-Leffler case \cite{Banderier}: $0<\alpha<1$, $\gamma=0, \beta>0, \theta>0$ }}  \strut\\
\hline
{\rm Beta}\left(\tfrac{\theta}{\alpha},\tfrac{\beta}{\alpha}\right)  & 
{\rm ML}(\alpha, \beta+\theta) & 
{\rm ML}(\alpha, \beta,0, \theta) 
& \Gamma(\beta+\theta) E^{\theta/\alpha}_{\alpha,\beta+\theta}(-x)
\TBstrut\\
\hline
\multicolumn{4}{c}{} \\
\multicolumn{4}{l} {\textrm{\footnotesize $0<\alpha<1$, $\beta-\alpha\gamma=1-\alpha \implies \gamma=(\beta-1)/\alpha+1, \theta>-\alpha\gamma$}}  \strut\\
\hline
{\rm Beta}\left(\tfrac{\theta}{\alpha}+\gamma,\tfrac{1}{\alpha}-1\right) & 
{\rm ML}(\alpha, \beta+\theta) & 
{\rm ML}(\alpha, \beta,\gamma,\theta) 
& \Gamma(\beta+\theta) E^{\gamma+\theta/\alpha}_{\alpha,\beta+\theta}(-x) 
\TBstrut\\
\hline
\multicolumn{4}{c}{} \\
\multicolumn{4}{l} {\textrm{\footnotesize Mittag-Leffler Markov Chain  
\cite{Goldschmidt22, GoldschmidtHaas, HoJamesLau, RembartWinkel2018, RembartWinkel2023, Senizergues}: 
$0<\alpha<1$, $\beta=\gamma=1, \theta>-\alpha, n\ge1$}}  \strut\\
\hline
{\rm Beta}\left(\tfrac{\theta+n-1}{\alpha}+1,\tfrac{1}{\alpha}-1\right) & 
{\rm ML}(\alpha, \theta+n) & 
{\rm ML}(\alpha, \theta+n-1)  & 
 \multicolumn{1}{c |} {\begin{array}{c}
\Gamma(\theta+n)\, \times 
\Tstrut\\
E^{1+(\theta+n-1)/\alpha}_{\alpha,\theta+n}(-x) \Bstrut\\
\end{array} }
\TBstrut\\
\hline
\multicolumn{4}{l} {\textrm{\normalsize $^\dagger{\rm ML}(\alpha, \beta, \gamma, \theta)$ has density: }}  \TBstrut\\
\multicolumn{4}{l} {\textrm{
$\quad p_{\alpha,\beta,\gamma,\theta}(t)
= \displaystyle\int_0^1 p_{\alpha,\beta+\theta}(t/u) \, {\rm beta}(u\vert \theta/\alpha+\gamma, \beta/\alpha-\gamma) \frac{du}{u}$  
 }}  \Bstrut\\
\end{array} 
\end{equation*}}
\caption{Theorem~\ref{thm:altML4pardistribution}: Distributions  of products of beta and  Mittag-Leffler variables}
\label{table:ML4parBetaML}
\end{table}
\newpage

%
\begin{table}[h!]
{\normalsize
\begin{equation*}
\begin{array}{| c | c  | c | } 
\cline{2-3} 
\multicolumn{1}{c|}{} & 
\multicolumn{1}{ c |}{\cellcolor{lightgray} \textrm{Density}} & 
 \multicolumn{1}{c |} {\cellcolor{lightgray}
\begin{array}{c}
\textrm{Laplace Transform} \\
\textrm{of density} 
\end{array}} 
\TBstrut\\
\hline 
& \multicolumn{1}{l |} { \textrm{\footnotesize Corollary~\ref{cor:MLconv2}: $0<\alpha<1; \lambda>0$ }} & \\
\textrm{\scriptsize 1}  & 
 \multicolumn{1}{c |} {
\begin{array}{c}
 \dfrac{1}{\lambda\pi}\,{\rm Im}\left \{\widetilde\rho_{1/\alpha-1}(e^{-i\pi\alpha}t^\alpha) \,  \widetilde\rho_{1,\lambda}(e^{-i\pi\alpha}t^\alpha)\right\}  \Bstrut\\
  = \dfrac{\sin\pi\alpha}{\pi}\, \dfrac{\lambda\, t^{\alpha-1}} {\lambda^2+2\lambda  t^\alpha \cos\pi\alpha+t^{2\alpha}}   \Bstrut\\
 \end{array}} &
E_{\alpha}(-\lambda x^\alpha) 
\TBstrut\\ 
 \hline 
\hline 
& \multicolumn{1}{l |} { \textrm{\footnotesize Corollary~\ref{cor:ML3parconv2}:  $0<\alpha<1, 0<\alpha\gamma<\beta; \,\lambda>0$}}  & \\
\textrm{\scriptsize 2}  & 
 \multicolumn{1}{c |} { 
\begin{array}{c}
 \dfrac{1}{\lambda^\gamma \pi}\,{\rm Im}\left \{\widetilde\rho_{\beta/\alpha-\gamma}(e^{-i\pi\alpha}t^\alpha) \, 
     \widetilde\rho_{\gamma,\lambda}(e^{-i\pi\alpha}t^\alpha)\right\} 
 \TBstrut\\
   = \dfrac{1}{\pi}\,{\rm Im} \dfrac{(e^{-i\pi}t)^{\alpha\gamma-\beta}}{(\lambda+  e^{-i\pi\alpha} t^\alpha)^\gamma}   \Bstrut\\
 \end{array}} &
 x^{\beta-1}E^\gamma_{\alpha,\beta}(-\lambda x^\alpha) 
\TBstrut\\ 
\hline 
\hline 
& \multicolumn{1}{l |} { \textrm{\footnotesize Corollary~\ref{cor:ML4parconv2}:  $0<\alpha<1, -\theta<\alpha\gamma<\beta; \,\lambda>0$}}  & \\
\textrm{\scriptsize 3}  & 
 \multicolumn{1}{c |} { 
\begin{array}{c}
 \dfrac{1}{\lambda^{\gamma+\theta/\alpha} \pi}\,{\rm Im}\left \{\widetilde\rho_{\beta/\alpha-\gamma}(e^{-i\pi\alpha}t^\alpha) \, 
     \widetilde\rho_{\gamma+\theta/\alpha,\lambda}(e^{-i\pi\alpha}t^\alpha)\right\} 
 \TBstrut\\
   = \dfrac{1}{\pi}\,{\rm Im} \dfrac{(e^{-i\pi}t)^{\alpha\gamma-\beta}}{(\lambda+  e^{-i\pi\alpha} t^\alpha)^{\gamma+\theta/\alpha}}   \Bstrut\\
 \end{array}} &
 x^{\beta+\theta-1}E^{\gamma+\theta/\alpha}_{\alpha,\beta+\theta}(-\lambda x^\alpha) 
\TBstrut\\ 
\hline
\multicolumn{3}{l}{\footnotesize  \textrm{Notes:}} \\
\multicolumn{3}{l}{\textrm{\footnotesize  $\rho_{\gamma,\lambda}(t)$  is the  gamma density with Laplace transform 
$\widetilde\rho_{\gamma,\lambda}(x)=\lambda^\gamma/(\lambda+x)^\gamma$ $(\gamma,\lambda>0)$} } \\
\multicolumn{3}{l}{\textrm{\footnotesize  $\rho_{\gamma}(t)$  is the  unnormalised power density with Laplace transform 
$\widetilde\rho_{\gamma}(x)=x^{-\gamma}$ } } \\
\end{array} \\
\end{equation*}}
\caption{Other  densities and their Laplace transforms. 
Only Case~1 is a probability  density, coincident with Case~2a of Table~\ref{table:MLDensities}. }
\label{table:OtherDensities}
\end{table}

\begin{table}[h!]
{\normalsize
\begin{equation*}
\begin{array}{ | c  | c |  c |} 
\hline
\multicolumn{1}{| c |}{\textrm{\footnotesize  Density 1}} & 
 \multicolumn{1}{| c |} {
 \begin{array}{c}   \textrm{\footnotesize  Generalised Positive Linnik density} \\ \textrm{\footnotesize(LT of density 1)} \end{array}}  &
 \multicolumn{1}{| c |} {\textrm{\footnotesize  LT of GPL density }}  
\TBstrut\\
\hline 
\multicolumn{3}{ l } { \textrm{\footnotesize Corollary~\ref{cor:ML4parconv2}:  
$0<\alpha<1, -\theta<\alpha\gamma=\beta; \,\lambda>0$}}   \Tstrut\\
\hline
\dfrac{1}{\pi}\,{\rm Im} \dfrac{\lambda^{\gamma+\theta/\alpha}}{(\lambda+  e^{-i\pi\alpha} t^\alpha)^{\gamma+\theta/\alpha}}  &
\begin{array}{c}  
\lambda^{\gamma+\theta/\alpha}   x^{\alpha\gamma+\theta-1}E^{\gamma+\theta/\alpha}_{\alpha,\alpha\gamma+\theta}(-\lambda x^\alpha) =   \Tstrut\\
\lambda^{\gamma+\theta/\alpha}  \displaystyle \int_0^\infty  f_\alpha(t \vert  u) \,dG(u\vert \gamma+\theta/\alpha,\lambda)  
\TBstrut\\ 
\end{array} &
\dfrac{\lambda^{\gamma+\theta/\alpha}}{(\lambda+s^\alpha)^{\gamma+\theta/\alpha}}
\TBstrut\\ 
\hline
\multicolumn{1}{| l | } { \textrm{\footnotesize Corollary~\ref{cor:PositiveLinnik}: $\gamma+\theta/\alpha=1$}} & 
\multicolumn{1}{ l | } { \textrm{\footnotesize Positive Linnik density}}  &  \\
 \multicolumn{1}{| c |} {
\begin{array}{c}
\displaystyle \int_0^\infty  t f_\alpha(t\vert \lambda u) \, dP_\alpha(u) =  \TBstrut\\
  \dfrac{\sin\pi\alpha}{\pi}\, \dfrac{\lambda\, t^{\alpha}} {\lambda^2+2\lambda  t^\alpha \cos\pi\alpha+t^{2\alpha}}   \TBstrut\\
 \end{array}} &
 \begin{array}{c}  
\lambda   x^{\alpha-1}E^{1}_{\alpha,\alpha}(-\lambda x^\alpha) \equiv -\dfrac{d}{dx}E_\alpha(-\lambda x^\alpha) \Tstrut\\
= \lambda \displaystyle \int_0^\infty  f_\alpha(t \vert  u) \, e^{-\lambda u}  \, du  
\TBstrut\\ 
\end{array} &
 \dfrac{ \lambda}{\lambda+s^\alpha}
\TBstrut\\ \hline
\end{array} \\
\end{equation*}}
\caption{Generalised Positive Linnik distribution ${\rm PL}(\alpha,\gamma+\theta/\alpha ,\lambda)$, 
with Positive Linnik distribution ${\rm PL}(\alpha,\lambda) \equiv {\rm PL}(\alpha,1 ,\lambda)$}
\label{table:Linnik}
\end{table}

\end{document}